\documentclass[12pt]{article}
\usepackage{amsthm}
\usepackage{graphicx}
\usepackage{amsmath}
\usepackage{xcolor}
\usepackage{amssymb}
\usepackage{colonequals}
\usepackage[pdftex]{hyperref}
\newtheorem{thm}{Theorem}[section]
\newtheorem{cor}[thm]{Corollary}
\newtheorem{lem}[thm]{Lemma}
\newtheorem{prop}[thm]{Proposition}
\theoremstyle{definition}
\newtheorem{defn}[thm]{Definition}
\theoremstyle{remark}
\newtheorem{rem}[thm]{Remark}
\theoremstyle{definition}
\newtheorem{ex}[thm]{Example}
\numberwithin{equation}{section}
\newcommand{\norm}[1]{\left\Vert#1\right\Vert}
\newcommand{\abs}[1]{\left\vert#1\right\vert}

\begin{document}

\title{Temperedness of measures defined by polynomial equations over local fields}%
\author{D. W. Taylor, V. S. Varadarajan, J. T. Virtanen, \\and D. E. Weisbart}%


\maketitle
\begin{abstract}
	We investigate the asymptotic growth of the canonical measures on the fibers of morphisms between vector spaces over local fields of arbitrary characteristic. For non-archimedean local fields we use a version of the {\L}ojasiewicz inequality (\cite{lojasiewicz1959}, \cite{hormander1958division}) which follows from Greenberg \cite{greenberg1966rational}, \cite{bollaerts1990estimate}, together with the theory of the Brauer group of local fields to construct definite forms of arbitrarily high degree, and to transfer questions at infinity to questions near the origin. We then use these to generalize results of H{\"o}rmander \cite{hormander1958division} on estimating the growth of polynomials at infinity in terms of the distance to their zero loci. Specifically, when a fiber corresponds to a non-critical value which is stable, i.e. remains non-critical under small perturbations, we show that the canonical measure on the fiber is tempered, which  generalizes results of Igusa and Raghavan \cite{igusa1978lectures}, and Virtanen and Weisbart \cite{virtanen2014elementary}.
\end{abstract}
\tableofcontents

\section{Introduction}
  Let $V$ be a finite dimensional vector space over $\mathbb{R}$ and $f\colon V \to \mathbb{R}$ a smooth non-constant function. In the physics and mathematics literature the measure denoted by $\delta (f-c)$ figures prominently; it is a measure living on the smooth part of the zero locus $Z(f-c)$ of $f-c$, $c\in \mathbb{R}$ \cite{gelfand1964}. Given $f$ and choices of Haar measures on $V$ and $\mathbb{R}$, $\delta(f-c)$ is uniquely defined for all $c$. Similarly if ${\bf f}=(f_1, f_2, \dots , f_r)\colon V\to \mathbb{R}^r$ is a  smooth map with $df_1\wedge\ldots\wedge df_r\neq 0$, for given Haar measures on $V$ and $\mathbb{R}^r$, there is a canonical measure on the smooth part of the common zero locus $Z({\bf f -c})=Z(f_1-c_1,f_2-c_2,\dots ,f_r-c_r)$ of the $f_i-c_i$ for all ${\bf c}=(c_1, c_2, \dots ,c_r)$. We denote this measure by $\mu_{{\bf f}, {\bf c}}$. In this context, the finiteness of $\mu_{{\bf f}, {\bf c}}$ around the singular points of $Z(f_1-c_1, f_2-c_2, \dots ,f_r-c_r)$, as well as the behavior at infinity of the extended measure, viewed as a Borel measure on $V$, are interesting questions. If the $f_i$ are polynomials and $Z({\bf f -c})$ is smooth, then it is natural to expect that $\mu_{{\bf f}, {\bf c}}$ is tempered. That is,
  
    	\begin{defn}[tempered measure]
    		Let $V$ be any finite dimensional $k$-vector space, $k$ a \emph{local field}. A Borel measure $\mu$ on $V$ is {\bf tempered} if  
    		\begin{displaymath}
    		\int_{V}(1+\norm{x}^2)^{-\alpha}d{\mu}(x)<\infty
    		\end{displaymath}
    		for some integer $\alpha$ (in any norm).
    	\end{defn}
  This is equivalent to saying that there are constants $A>0, b\ge 0$ such that 
  \[
  \mu(B_R)\le AR^b\qquad \label{G} \tag{G}
  \]
  for all $R\ge 1, B_R$ being the closed ball in $V$ of radius $R$ and center ${\bf 0}$ (in any norm).
 
  In \cite{igusa1978lectures} Igusa and Raghavan proved that if $k=\mathbb{R}$ and $f$ is a non-constant polynomial on $V$ and $c\in \mathbb{R}$ is a non-critical value of $f$, i.e., the locus $Z(f-c)$ is smooth, then $\mu_{f,c}$ is tempered, and further that the growth estimate (\ref{G}) for the measure is uniform in a neighborhood of $c$; here we must remember that by the algebraic Sard's theorem (proposition \ref{prop:2.4}), $f$ has only finitely many critical values, so that every non-critical value $c$ has neighborhoods consisting only of non-critical values. 
    
  The measures $\mu_{f,c}, \mu_{\bf {f,c}}$ can be defined over any local field.\footnote{{Throughout this paper by local field we mean a locally compact non-discrete field of any characteristic, other than the complexes; measure theoretic questions over the complexes usually reduce to the reals, and so we do not treat the case of complexes separately.}}
  In \cite{igusa1978lectures} Igusa and Raghavan define the measures $\mu_{f,c}$ for any local field but do not consider their behavior at infinity, the reason being that over a non-archimedean field they were concerned only with integrating Schwartz-Bruhat functions (i.e. compactly supported complex-valued locally constant functions). However the work of Harish-Chandra \cite{Har2} shows the necessity as well as utility of working with locally constant functions that do not vanish outside a compact set. The question of extending the results of \cite{igusa1978lectures} to the non-archimedean case and for $r>1$ is certainly a natural one. In \cite{virtanen2014elementary} the measures $\mu_{f,c}$ were shown to be tempered when $f$ is a non-degenerate quadratic form and $c \not=0$; moreover for the case $c=0$ the locus $Z(f)$ has $0$ as its only singularity, and it was shown that the measure $\mu_{f,0}$ is finite in the neighborhood of $0$ if $\dim V \geq 3$, and the extended measure is tempered in $V$. The work of \cite{virtanen2014elementary} was motivated by physical questions arising in the theory of elementary particles over $p$-adic spacetimes. In this paper we generalize the results of \cite{igusa1978lectures} and \cite{virtanen2014elementary} to the measures $\mu_{\bf {f,c}}$ where the $f_i(1\le i\le r)$ are polynomials on a vector space $V$ over a local field $k$, with $\dim (V)=m$ and $df_1\wedge df_2\wedge \dots \wedge df_r\not=0$, so that $m\ge r$. Note that for $r>1$ and $k=\mathbb{R}$ this question is already more general than the one treated in \cite{igusa1978lectures}.
 
  We now describe our main result using the above notation. Let ${\bf f}\colon V\to k^r$ be the polynomial map whose components are the $f_i$, with $df_1\wedge\cdots\wedge df_r \not\equiv 0$. A point $x\in V$ is called a critical point (CP) of ${\bf f}$ if the differentials $df_{i, x}$ are linearly dependent. We write $C({\bf f})$ for the set of critical points of ${\bf f}$; the image ${\bf f}(C({\bf f}))$ in $k^r$ is called the set of critical values of ${\bf f}$, and is denoted by $CV({\bf f})$. By the algebraic Sard's theorem (proposition \ref{prop:2.4}) one knows that {\it in characteristic zero} the Zariski closure in $k^r$ of $CV({\bf f})$ is a {\it proper} algebraic subset of $k^r$. A point $c\in k^r$ is called {\it stably non-critical}\label{stabNonCrit} if it has an open neighborhood (in the $k$-topology) consisting only of non-critical values. This is the same as saying that the fibers above points sufficiently close to $c$ are smooth. If $k$ has characteristic zero stably non-critical values exist and form a non-empty open set in $k^r$ whose complement in the image of ${\bf f}$ has measure $0$. Then the following is our main result. For $r=1$ and $k=\mathbb{R}$ it was proved in \cite{igusa1978lectures}. Note that in this case the characteristic is $0$ and there are only finitely many critical values and so every non-critical value is stably non-critical.
 
  \begin{thm}\label{thm:1.1}
  	{ Fix ${\bf f}$ and write $\mu _{\bf c}=\mu _{\bf f,c}$. Suppose ${\bf c}$ is stably non-critical. Then $\mu_{\bf c}$ is tempered and there are constants $A>0, \gamma\ge 0$ such that for all ${\bf d}$ in an open neighborhood of ${\bf c}$
  	$$
  	\mu _{\bf d}(B_R)\le A R^{m-r+\gamma}\qquad (R\ge 1, {\bf d}\in U).
  	$$
  	Suppose $k$ has characteristic $0$; then stably non-critical values form a non-empty dense open set whose complement in the image of $\mathrm{\bf f}$ has measure 0; for $r=1$, the critical set is finite and all non-critical values are stably non-critical.}
  \end{thm}
  \begin{rem}
  	In view of the failure of Sard's theorem over characteristic $p>0$ (see section \S 2.3), we do not know if stably non-critical values of ${\bf c}$ always exist when $k$ is a local field of positive characteristic.
  \end{rem}
  \begin{rem}
  	The results and ideas in the paper lie at the interface of analysis of geometry over local fields and are motivated by the themes from quantum theory over $p$-adic spacetimes. We do not know what, if any, are the arithmetic consequences of our results.
  \end{rem}
  As an application of our theory we prove that if $k$ has characteristic $0$, the orbits of regular semi-simple elements of a semi-simple Lie algebra over $k$ are closed, and the invariant measures on them are tempered, at least when the invariants of the adjoint group are \emph{freely generated} over $k$; this property can always be ensured by passing to a finite extension of $k$. For $k=\mathbb{R}$ this is a special case of a result of Harish-Chandra in \cite{Har1}, who proved it without any condition on the structure of the ring of invariants over $\mathbb{R}$.
  
\section{Canonical measures on level sets of polynomial maps}

\subsection{Canonical measures on the fibers of submersive maps}
 The construction below is well-known and our treatment is a very mild variant of Harish-Chandra's \cite{Har3} for the case $k=\mathbb{R}$ (see also \cite{Var1}). Serre's book \cite{serre2009lie} is a good reference for the theory of analytic manifolds and maps over a local field of arbitrary characteristic. (All of our manifolds are second countable.)
 
 \begin{lem}\label{thm:2.2}
 	{Let $V, W$ be vector spaces of finite dimension $m, r$ respectively, and $L\colon V\to W$ be a surjective linear map. Let $U={\ker L}$. Let $\sigma,\ \tau$ be exterior forms on $V, W$ of degrees $m, r$ respectively, with $\tau\not=0$. Then there exists a unique exterior $(m-r)$-form $\rho$ on $U$ such that if $\{u_1,u_2, \dots , u_{m-r}\}$ is a basis for $U$, then
 		$$
 		\rho(u_1, u_2, \dots ,u_{m-r})={\sigma (u_1, \dots ,u_{m-r}, v_1,\dots ,v_r)\over \tau (Lv_1,\dots ,Lv_r)}
 		$$
 		where $v_i \in V$ are such that $\{u_1, \dots ,u_{m-r}, v_1,\dots ,v_r\}$ is a basis for $V$.}
 \end{lem}

 \begin{proof}
 	For fixed $v_i$ it is obvious that this defines an exterior $(m-r)$-form on $U$. Its independence of the choice of the $v_i$ is easy to check.
 \end{proof}
 
 We write $\rho=\sigma /\tau$. Note that this definition is relative to $L$.
 
 \begin{thm}\label{lem:2.2}
 	{Let $k$ be a local field of arbitrary characteristic and $M, N$ be analytic manifolds over $k$ of dimensions $m, r$ respectively, and $\pi \colon M\to N$ be an analytic map, surjective, and submersive everywhere. Let $\sigma_M$ (resp.$\tau_N$) be an analytic exterior $m$-form (resp. $r$-form) on $M$ (resp. $N$), with $\tau_N\not=0$ everywhere on $N$. Then there is a unique analytic exterior form $\rho\colonequals\rho_{M/N}$ on $M$ such that for any $y\in N$, the pull back of $\rho$ to the fiber $\pi ^{-1}(y)$ is the exterior $(m-r)$-form $x\mapsto \sigma_x/\tau_y$ relative to $d\pi_x\colon T_x(M)\to T_y(N)$.}
 \end{thm} 
 \bigskip\noindent
\begin{proof}
 	The pointwise definition of $\rho$ is clear after the preceding lemma. For analyticity we use local coordinates around $x$ and $y=\pi(x)$, say $x_1, \dots x_m$ such that $\pi$ is the projection $(x_1, \dots ,x_m)\to (x_1,\dots ,x_r)$. Then $$\sigma_M=s(x_1, \dots ,x_m)dx_1 \dots dx_m, \tau=t(x_1, \dots ,x_r) dx_1\dots dx_r,$$  and $$\rho=\big ( s(x_1, \dots ,x_m)/t(x_1, \dots ,x_r) \big) dx_{n+1}\dots dx_{m}.$$
\end{proof}

\begin{rem}
	Let $s_M$ (resp. $t_N$) be the measures defined on $M$ (resp. $N$) by $|\sigma_M|$ (resp. $|\tau_N|$). We denote by $r_{M/N,\,y}$ the measures defined on $\pi^{-1}(y)$ by $\abs{\rho}$. The smooth functions in the non-archimedean case are the locally constant functions. Then, we have \cite{Har3}
	$$
	\int _M\alpha \ ds_M=\int_N f_\alpha \ dt_N,\quad f_\alpha (y)=\int _{\pi^{-1}(y)}\alpha\  dr_{M/N,\, y}
	$$
	for all smooth compactly supported complex-valued functions $\alpha$ on $M$. 
\end{rem}

It is easy to show, using partitions of unity that the map $\alpha \mapsto f_\alpha$ is surjective, and continuous when $k=\mathbb{R}$. This gives rise to an injection of the space of distributions on $N$ into the space of distributions on$M$, say $T\mapsto T^*$. Then $r_{M/N,y}= \delta(y)^*$, $\delta(y)$ being the Dirac distribution at $y\in N$. Replacing $\delta(y)$ by its derivatives, we get distributions on $M$, supported by $\pi^{-1}(y)$. If $F$ is a locally integrable function on $N$, it defines a distribution on $N$, say $T_F$, and $T_F^*$ is $T_{F\circ \pi}$ where $F\circ \pi$ is a locally integrable function on $M$. Thus the map $T\mapsto T^*$ is the natural extension of the map $F\mapsto F\circ \pi$ from the space of locally integrable functions on $N$ to the corresponding space on $M$. The map $T\mapsto T^*$ plays a fundamental role in Harish-Chandra's theory of characters on real semi-simple Lie groups (\cite{Har3}).
Finally, in algebro-geometric terminology, $\rho$ above is the top \emph{relative} exterior form.
 
We shall now apply this result to polynomial maps ${\bf f} \colon V\to k^r$ where $V$ is a vector space of finite dimension $m$ over a local field $k$ of arbitrary characteristic such that $df_1\wedge \dots df_r\not\equiv 0$ on $V$, the $f_i$ being the components of ${\bf f}$; let $V^\times$ be the set of points where this exterior form is non-zero in $V$, so that $V^\times $ is non-empty Zariski open in $V$; let $N({\bf f})={\bf f}(V^\times)$. Clearly $m\ge r$ and $N({\bf f})$ is non-empty open (in the $k$-topology) in $k^r$. Then, by theorem \ref{lem:2.2} with $M=V^\times, N=N({\bf f})$, we have a measure $\mu_{\bf c}$ for ${\bf c}\in N$ on $L_{\bf c}^\prime \colonequals L_{\bf c}\cap V^\times$ where $L_{\bf c}$ is the level set
 \begin{equation}
   L_{\bf c}= Z(f_1-c_1, \dots ,f_r-c_r)=\{x\in V | f_1 (x)=c_1, \dots ,f_r(x)=c_r\}. \label{eqn:Lc}
 \end{equation}
 
 Exactly as before, we may view the $\mu_{\bf f, c}$ as distributions living on $L_{\bf c}'$ which is all of $L_{\bf c}$ if ${\bf c}$ is a non-critical value. The derivatives of $\mu_{\bf f, c}$ with respect to the differential operators of $k^m$ (when $k=\mathbb{R}$) then yield distributions supported by $L_{\bf c}$. Examples of such distributions have important applications (\cite{gelfand1964},\cite{Kolk}) in analysis and physics

 Fix a non-critical value ${\bf c}$ of ${\bf f}$. Let $J=\{i_1<i_2, \dots , < i_r\}$ be an ordered subset of $r$ elements in $\{1,2, \dots ,m\}$. Let
 
 \begin{equation}
   \partial _J \colonequals{\partial (f_1,\dots ,f_r)\over \partial (x_{i_1}, \dots , x_{i_r})}. \label{eqn:partialJ}
 \end{equation}
 
 Then $L_{\bf c}$ is smooth and $L_{\bf c}=\bigcup _J L_{{\bf c}, J}$
 where the sum is over all sets $J$ as above and 
 \begin{equation}
   L_{{\bf c}, J}\colonequals\{x\in L_{\bf c} | \partial _J(x)\not=0\}.\label{eqn:LcJ}
 \end{equation}
 Locally on $L_{{\bf c}, J}$, $(f_1,\dots ,f_r, y_{1},\dots y_{m-r})$ is a new coordinate system, the $y_j$ being some enumeration of the $x_i (i\not=i_\nu)$. Obviously $dy_1\dots dy_m=\varepsilon \partial _J(x) dx_1\dots dx_m,$
 where $\varepsilon$ is locally constant and equal to $\pm 1$. Another way of interpreting this formula is the following: if $\pi_J$ is the projection map from $L_{{\bf c}, J}$ that takes $x$ to $(y_{1}, \dots y_{m-r})$, then $\pi_J$ is a local analytic isomorphism and
 \begin{equation}
   \rho_{\bf c}=\varepsilon {1\over \partial_J(x)} \pi_J^\ast (dy_{1}\dots dy_{m-r}) \label{eqn:measure}
 \end{equation}
 where $\varepsilon$ is locally constant and $\pm 1$-valued. Hence to control the growth of the measure defined by $\abs{\rho}$ at infinity, we must find {\it lower bounds} of the $\norm{\partial _J(x)}$ on $L_{{\bf c}, J}$
 for $||x||\ge 1$. Let
 $$
 \nabla _r(x)= \big (\partial_J(x)).
 $$
 We call $\nabla_r$ the generalized gradient of $ (f_1,\ldots, f_r).$ Then we must find lower bounds for $\norm{\nabla_r(x)}\colonequals \max_{J}\norm{\partial _J(x)}$ for $||x||\ge 1$ on $L_{{\bf c}, J}$. In this quest we follow \cite{igusa1978lectures}, and our techniques force us to assume ${\bf c}$ to be stably non-critical. For $r=1$, $\nabla_1$ is just the gradient $\nabla$, and \cite{igusa1978lectures} reduces the issue of the lower bounds for the gradient field by replacing $\nabla f$ (for $k=\mathbb{R}$) by $\sum _{1\le j\le m}|\partial_jf|^2$, where $\partial_jf=\partial f/\partial x_j$. For non-archimedean local fields and for $r>1$ we have to replace the sum of squares by suitable {\it definite} forms whose degrees will grow with $m$. Igusa and Raghavan find lower bounds for $\abs{\nabla}$ using H{\"o}rmander's inequalities \cite{hormander1958division} over $\mathbb{R}$. We generalize H{\"o}rmander's inequalities to any local field and use them with the existence of definite forms of sufficiently high degree to get lower bounds for $\norm{\nabla_r}$ on the level sets $L_{{\bf c}, J}.$ 
 
 The H{\"o}rmander inequalities over $\mathbb{R}$ are of two types: H1 and H2. H1 is local and is essentially the {\L}ojasiewicz inequality \cite{lojasiewicz1959}; H{\"o}rmander derives H2 from H1 by inversion. Over non-archimedean $k$, H1 turns out to be a consequence of a Henselization lemma of Greenberg \cite{greenberg1966rational}, as observed in \cite{bollaerts1990estimate}. The reduction of H2 to H1 is more subtle in the non-archimedean case.
 We prove it by embedding $V$ in a division algebra $D$, central over $k$, prove H2 for $D$, and then deduce $H2$ for $V$. The descent from $D$ to $V$ is elementary. To prove H2 in $D$ we use the map $x\mapsto x^{-1}$ on $D\setminus\{0\}$ to reduce H2 to H1. The existence of central division algebras over $k$ of arbitrarily high dimension is non-trivial and follows from the theory of the Brauer group of $k$. The lower bounds of $\nabla_r f$ obtained from these arguments allow us to prove that when ${\bf c}$ is a \emph{stably non-critical} value of ${\bf f}$, $\mu_{{\bf f},{\bf c}}(B_r)=O(R^{m-r+\gamma})$ for some $\gamma\geq 0$, uniformly near ${\bf c}.$ We do not know if we can take $\gamma =0$ always. If $\norm{\nabla_r {\bf f}}$ is bounded away from zero at infinity on $L_{\bf c}$, then it is obvious that we may take $\gamma = 0$; but $\inf \norm{\nabla_r{\bf f}}$ may be zero on $L_{\bf c}$. (See section 7.2)
\subsection{Algebraic Sard's theorem in characteristic $0$ for polynomial maps}\label{sec:2.2}
Let $V$ be a vector space over $k$ of finite dimension $m$. Recall the definitions of $C(\mathbf{f})$ and $CV(\mathbf{f})$.

\begin{prop}\label{prop:2.4}
	Let $k$ be of characteristic 0. The Zariski closure, $Cl(CV(\mathbf{f}) )$ is a proper subset of $k^r$; in particular, if $r=1$, then $CV(\mathbf{f})$ is finite.
\end{prop} 

\begin{proof}
	Fix a basis of $V$ so that $V\simeq k^m$. The field generated by the coefficients of the $f_j$, say $k_1\supset k$, can be embedded in $\mathbb{C}$. It is thus enough to prove lemma \ref{prop:2.4} over $\mathbb{C}$ itself, where it is just the statement that the fibers of ${\bf f}$ are generically smooth.  
	Over $\mathbb{C}$ this is essentially Sard's lemma for affine algebraic varieties treated by Mumford \cite{mumford1995algebraic} . 	
\end{proof}
\subsection{Analytic Sard's theorem in characteristic $p>0$}
In characteristic $p>0$, the algebraic Sard's lemma fails abysmally \cite{MO2015} (p.179) over algebraically closed fields. Indeed, let $f$ be a polynomial in two variables $X, Y$ giving rise to a map $K^2\longrightarrow K$ where $K$ is algebraically closed and of characteristic $p>0$, for example,
$$
f=X^{p+1}+X^pY+Y^p.
$$ 
Then the gradient of $f$ vanishes precisely on the $Y$-axis, and $f$ on the $Y$-axis is the map $y\mapsto y^p$ which is surjective. So the image of the singular set is all of $K$, and every fiber has a singular point. But if we replace $K$ by a local field, then $y\mapsto y^p$ is {\it not} surjective, and in fact the image under $f$ of the singular set is $k^p$ which  is a  closed proper subset of $k$ (in the $k$-topology), and is of measure zero in $k$. Thus the generic fiber (in the $k$-topology) is smooth in $k$.

We shall now consider the situation over local fields of characteristic $p>0$.
\begin{thm}\label{thm:2.5}
	Let $X,Y$ be analytic manifolds over a local field $k$ of characteristic $p>0$, of dimensions $m,r$ respectively. Let $f\colon X\to Y$ be an analytic. 
	Let $C$ be the critical set for $f$. Then $f(C)$ has measure zero in $Y$ if $p>m-n+1$.
\end{thm}
	
\begin{proof} 
	The proof that $f(C)$ has measure zero in $Y$ when $p>m-r+1$ is a minor adaptation of \cite{guillemin1974difierential}\footnote{From Sard \cite{Sard42} we know that when $k=\mathbb{R}$ and the map is of class $C^{(a)}$ $(a>0)$, $f(C)$ has measure zero when $a>m-r$. Now, when $k$ has characteristic $p>0$, the derivatives of $f$ are not enough to determine the coefficients of the power series expansion of $f$ whose order is greater than $p-1$. So there is an analogy with the case of $C^{(p-1)}$ over the reals, suggesting that over $k$ the condition $p>m-r+1$ would be sufficient to guarantee that $f(C)$ is a null set. This suggestion, which leads to theorem \ref{thm:2.5}, is due to Professor Pierre Deligne, which we gratefully acknowledge.}, needed because  we have an additional restriction on $p$.
	
	The result is local and so we may take $X$ to be a compact open set $U\in k^m$. We use induction on $m$. We define the filtration $C=C_0\supset C_1\supset \ldots\supset C_{p-1}$, where $C_s \ (1\leq s \leq p-1) $ is the set where all derivatives of the components of $f$ of order $\le s$ vanish. The sets $C,C_s$ are compact while $C_s\setminus C_{s+1}$ is locally compact and second countable, hence a countable union of compact sets. So $f(C), f(C_s)$ are compact, and $f(C_s\setminus C_{s+1})$ is a countable union of compact sets.
	
	The inductive proof that $f(C\setminus C_1)$ is a null set reduces to the case when $(m,r)$ becomes $(m-1,r-1)$. Since $m-r=(m-1)-(r-1)$, the condition on $p$ remains the same and induction applies.
	
	The inductive proof that $f(C_s\setminus C_{s+1})$ is a null set reduces to the case when $(m,r)$ becomes $(m-1,r)$. Since $p>m-r+1>(m-1)-r+1$, induction applies again.
	
	It remains to show that $f(C_{p-1})$ is a null set when $p>m-r+1$. We shall show actually that $f(C_{p-1})$ is a null set when $p>m/r$. This is enough since $m/r \le m-r+1$. This is a local result and so we may work around a point of $C_{p-1}$ which can be taken to be the origin. We use the max norm on $k^m$ and $k^r$ so that the norms take values in $q^\mathbb{Z}$ where $q>1$ is the cardinality of the residue field of $k$. By scaling, if necessary, we may assume that all components of $f$ are given by power series expansions, absolutely convergent on the ball $B(q)\colonequals\left\{ x\in k^m\, |\, \norm{x} \le q\right\}$. Note that $B(1)=R^m$ where $R$ is the ring of integers of $k$. In order to estimate the growth of these series we need a lemma.
		\begin{lem} 
			{Let $g$ be an analytic function on $B(q)$ given by an absolutely convergent power series expansion about $0$ on $B(q)$.  Let $D$ be the set in $B(1)$ where $\partial^{\beta} f=0$ for all $\beta$ with $\abs{\beta}\le p-1$. Then we have 
				$$
				\abs{g(x+h) -g(x)} \le A\norm{h}^{p}
				$$
				uniformly for $x\in D, \norm{h} \le 1 \le q-1$, the constant $A >0$ depending only on $g$.}
		\end{lem} 
			
			\begin{proof} 
				We use \cite{serre2009lie}, pp. 67-75. We have
				$$
				  g(x)=\sum _\alpha c_\alpha X^\alpha,\qquad \sum_\alpha |c_\alpha|=A<\infty.
				$$
				For $x\in B(1)$ we have $g(x+h)=\sum _\beta\Delta^\beta g(x)h^\beta$ where
				\begin{displaymath}
					\Delta^{\beta}g(x) =\sum_{\alpha \ge \beta} c_\alpha\binom{\alpha}{\beta}x^{\alpha-\beta}, \qquad \beta ! \Delta^\beta g(x)=\partial^\beta g(x).
				\end{displaymath}
				Then $|\Delta^\beta g(x)|\le A$ on $B(1)$. If $x\in D$, $\norm{h}\le 1\le q-1$ then $x+h\in B(1)$. Moreover, for $\abs{\beta}\le p-1$, $\beta! \Delta^\beta g(x)=0$ so that $\Delta^\beta g(x)=0$.
				Hence,
				\begin{displaymath}
					g(x+h)= g(x)+ \sum_{\abs{\beta} \ge p }(\Delta^\beta g )(x) h^\beta
				\end{displaymath}
				But, for $y\in B(1)$,
				\begin{displaymath}
					\abs{\Delta^{\beta} g(y)} \le \sum\abs{c_\alpha} = A.
				\end{displaymath} 
				So, 
				\begin{displaymath}
					\abs{g(x+h)-g(x)} \le A\norm{h}^p \qquad (x\in D, \norm{h} \le 1),
				\end{displaymath}
				proving the lemma.
			\end{proof}
			
			We now divide $B(1)^m$ into very small ``cells''. Let $P$ be the maximal ideal in $R$. Let $N$ be any integer $\ge 1$. Then $B(1)$ is the disjoint union of $q^N$ cosets of $P^N$ each of which is a compact open set that has diameter $\le q^{-N}$ and volume $q^{-N}$. This gives a partition of $B(1)^m$ into $q^{mN}$ compact open sets (``cells'') of diameter $\le q^{-N}$ and volume $q^{-Nm}$. By the above lemma, if $x, x+h \in D$ and are in one of these cells, say $\gamma$, then 
			\begin{equation}
   			  \norm{f(x+h)-f(x)}\le A\norm{h}^{p} \le A q^{-Np}, 
			\end{equation}
			where $A$ is a constant independent of $x$. Hence, $f(\gamma)$ is contained in a set of diameter $\le q^{-Np}$ and hence volume $\le q^{-Npr}$. Thus $f(D\cap C_{p-1})$ is enclosed in a set of volume $\le q^{mN-Npr}=q^{-N(pr-m)}$. If $p>m/r$ this expression goes to $0$ as $N\to \infty$, and we are done.
		\end{proof}
		
		\begin{rem} 
			If ${\bf f}=(f_1,\dots ,f_r)$ is a polynomial map of $k^m$ into $k^r$ such that $df_1\wedge \dots \wedge df_r\not=0$, then ${\bf f}(k^m)$ is open and Sard's theorem shows that almost every fiber of $f$ is smooth in $k$. So there are always non-critical values. Whether some of them are stable is not known to us.
		\end{rem}
		
		\begin{rem}
			When $r=1$, the above condition reduces to $p>m$. Both this condition and the fact that when $m \ge p+1$ it is possible that the image of the critical set can be all of $k$ were communicated to us by Professor Pierre Deligne. We are grateful for his generosity and for giving us permission to discuss his example.
		\end{rem}
		
		\begin{ex}[Deligne]
			We take $m=p+1$ with coordinates $y,x_1,\ldots,x_p$. The field $k\colonequals\mathbb{F}[[t]][{1}/{t}]$ where $\mathbb{F}$ is a finite field of characteristic $p$, is a local field of characteristic $p$. Then $k$ is a vector space of dimension $p$ over $k^{(p)}\colonequals\{x^p\,|\, x\in k \}$. Let $(a_i)_{1\le i \le p}$ be a basis for $k/k^{(p)}$, for instance $a_i=t^{i-1}$, $( 1 \le i \le p)$. Consider, for an integer $n>1$, prime to $p$,
			\begin{displaymath}
				f=y^n + a_1 x_1^p + \cdots + a_p x_p^p.
			\end{displaymath}
			Then the critical locus is given by $y=0$. Its image under $f$ is obviously all of $k$. If we do not insist that $df\not\equiv 0$, we can omit $y$ so that $f$ maps the critical set $k^p$ onto $k$.
			
			This example is easily modified for the case $r>1$. We consider $k^{p+r}$ with coordinates $y_1,\ldots,y_{r-1}, y,x_1\ldots, x_p$ and take the map $f\colon k^{p+r}\to k^r$ defined by 
			\begin{displaymath}
				f\colon (y_1,\ldots,y_{r-1},y,x_1\ldots, x_p)\mapsto (y_1,\ldots,y_{r-1},y^n+\sum_{i=1}^{p} a_i x_i^p),
			\end{displaymath}
			where the notation is as before. The critical set is again given by $y=0$, and the map restricted to this set is 
			\begin{displaymath}
				f\colon (y_1,\ldots,y_{r-1},0,x_1\ldots, x_p)\mapsto (y_1,\ldots,y_{r-1},\sum_{i=1}^{p} a_i x_i^p),
			\end{displaymath}
			whose range is $k^r$. Exactly as before, if we omit $y$, we get a map where $df_1\wedge\cdots \wedge df_r$ is zero but ${\bf f}$ maps the critical set $k^{p+r-1}$ onto $k^r$.
		\end{ex}

\section{Construction of definite forms and their associated norms}
As mentioned in the remarks at the end of section 2.1 we begin by discussing the construction of definite forms in an arbitrary number of variables over $k$.
\begin{prop}\label{prop:3.1}
	Let $V$ be a finite dimensional vector space over a local field. If $k=\mathbb{R}$, and $\nu(x)$ is a positive definite quadratic form on $V$, then $\abs{\nu(x)}^{1/2}$ is a norm on $V$. If $k$ is non-archimedean, and $r$ is an integer such that $r^2 \geq m$, then there is a homogeneous polynomial $\nu\colon V\to k$ of degree $r$ such that 
	\begin{enumerate}
		\item[a)] $\nu$ is definite, i.e., for $x\in V$, $\nu(x)=0$ if and only if $x=0.$
		\item[b)] $\abs{\nu(x)}^{1/r}$ is a a non-archimedean norm on $V$. 
	\end{enumerate}
\end{prop}

\begin{proof}
	We deal only with the case of non-archimedean $k$. By the theory of the Brauer group of $k$ (\cite{weil2013basic}, ch. XII, theorem 1 and its corollary) we can find a division algebra $D$ over $k$ which is central over $k$ and $\dim_k (D)=r^2$. Since $V\hookrightarrow D$, it is enough to prove the proposition for $V=D$. The advantage is that we can use the algebraic structure of $D$.
	
	Let $\nu$ be the reduced norm (\cite{weil2013basic}, ch. IX, proposition 6) of $D$. Then, $\nu\colon D\to k$ is a homogeneous polynomial function on $D$ of degree $r$, and $\nu(x)^r=\det(\lambda(x))$ where $\lambda(x)$ is the endomorphism $y\mapsto xy$ of $D$. Note that $\det(\lambda)$ is a polynomial function on $D$ with values in $k$, homogeneous of degree $r^2$. As $\lambda(x)$ is invertible for any $x\neq 0$ in $D$, $\det(\lambda(x) )$ and hence $\nu(x)$, is non-zero for $x\neq 0$ in $D$. Hence, $\nu$ is a definite form of degree $r$ on $D$. It remains to prove that $N(x)\colonequals\abs{\nu(x)}^{1/r}$ is a non-archimedean norm on $D$. This reduces to showing that 
	$N(1+u)\leq 1 $ if $u\in D$ and $N(u) \leq 1 $, or equivalently, that $\abs{\lambda(1+u)} \leq 1$ if $u\in D$ and $\abs{\lambda(u)} \leq 1,$ which follows from \cite{weil2013basic}, ch. I, section 4.
\end{proof}

\begin{rem}
	Actually, $\nu(x)^r=\det\lambda(x)$ will serve our purposes as well and is obviously a homogeneous polynomial of degree $r^2$, Then $\abs{\nu(x)}^{1/r}=\abs{\det\lambda(x)}^{1/r^2}$. We introduced $\nu$ because it is of smaller degree and this may be of use in other contexts. 
\end{rem}

\section{H\"{o}rmander's inequalities over non-archimedean local fields}
Let $V$ be a finite dimensional vector space over a local, non-archimedean field, $k$, with its canonical norm $\abs{\cdot}$. Let $ \norm{\cdot }$
be a non-archimedean norm on $V$. We may assume that the norms on $k$ and $V$ take values in the set $\{0,q^{\pm 1}, q^{\pm 2},\ldots \},$ where $q$ is the cardinality of the residue field of $k$. Also, let $f\colon V\rightarrow k$ be a polynomial function, and let $Z(f)$ denote its zero locus. For $x \in V$ and nonempty $E\subset V$ let ${\rm dist} (x, E) \colonequals \inf_{y\in E}\norm{x-y} $.

\begin{thm}[H1]\label{thm:H1}
	Let $f\colon V\rightarrow k$ be a polynomial function on $V$. Suppose that $Z(f)\neq \emptyset$. Then there exist constants $C>0,$ $\alpha \geq 0$ such that 
	\begin{equation}\label{eq:4.1}
	\lvert f(x) \rvert \geq C\cdot {\rm dist} (x,Z(f))^\alpha
	\end{equation}
	for all $x\in V$ with $\norm{x} \leq 1.$
\end{thm}

\begin{thm}[H2]\label{thm:H2}
	Let $f\colon V\rightarrow k$ be a polynomial function, $Z(f)$ as above. Then
	\begin{enumerate} 
		\item[a)] if $Z(f) = \emptyset$, then there exist constants $C>0$ and $ \beta \geq 0$ such that
		\begin{equation}
		\abs{f(x)} \geq C\cdot \frac{1}{\norm{x}^{\beta} } \qquad (x\in V, \norm{x}\geq 1)
		\end{equation}
		\item[b)] if $Z(f) \neq \emptyset$, then there exist constants $C>0$ and $\alpha, \beta \geq 0$ such that
		\begin{equation}
		\norm{f(x)} \geq C\cdot \frac{{\rm dist}(x,Z(f) )^\alpha }{\norm{x}^\beta}\qquad (x\in V, \norm{x}\geq 1)
		\end{equation}
	\end{enumerate}
\end{thm}
\begin{rem}\label{rem:4.4}
	Theorem \ref{thm:H1} and \ref{thm:H2} were proved by H\"{o}rmander \cite{hormander1958division} when $k=\mathbb{R}$. Also, H1 is a special case of the {\L}ojasiewicz inequality for $f$ a real analytic function \cite{lojasiewicz1959}.
\end{rem}

In proving H1 we may assume that $V=k^m$ and $f\in R[x_1,\ldots,x_m]$, $R$ being the ring of integers in $k$. Let $P\subset R$ be the maximal ideal of $R$. Suppose that $Z(f)\neq \emptyset$ but $Z(f)\cap R^m= \emptyset$. Then there exists a constant $b>0$ such that $\abs{f(x)}\geq b>0$ for $x\in R^m$. On the other hand, as $R^m$ is compact, there exists $b_1>0$ such that ${\rm dist}(x,Z(f)) \leq b_1$ for all $x\in R^m$. Hence $\abs{f(x)}\geq b b_1^{-1} b_1 \geq b b_1^{-1}{\rm dist}(x,Z(f))$ for all $x\in R^m$. Hence we may assume in addition that $Z(f)\cap R^m \neq\emptyset$ in the proof of H1.

\subsection{Proof of H1: $k$ non-archimedean} 
We follow Greenberg (\cite{greenberg1966rational}), specialized to the case of a single polynomial.
\begin{proof}
	By theorem 1 of \cite{greenberg1966rational}, applied to the single polynomial $f$, we can find integers, $N,c\geq 1$ and $s\geq 0$ such that if $\nu \geq N$ and $f(x)\equiv 0\; (\text{mod }P^\nu)$, and $x\in R^m$, then there exists $y\in R^m$ such that $f(y)= 0$ and $x_i-y_i \equiv 0 \; (\text{mod }P^{[\nu/c]-s})$ for all $i$.
	
	Assume $\abs{ f(x) } =q^{-(N+\ell)} $, $\ell \geq 0$. Then, there exists $y\in Z(f)\cap R^m$ such that 
	\begin{equation*}
	\norm{x-y} \leq q^{-[(N+\ell)/c]+s} \leq q^{-[(N+\ell)/c -1]+s}\leq q^{s+1}| f(x)|^{1/c} ,
	\end{equation*}
	which implies that 
	\begin{equation*}
	{\rm dist}(x,Z(f)\cap R^m) \leq q^{s+1} \lvert f(x) \rvert^{1/c},
	\end{equation*}
	so that
	\begin{equation*}
	  \lvert f(x) \rvert  \geq \frac{{\rm dist}(x, Z(f)\cap R^m)^c}{q^{c(s+1)}}
	  \geq \frac{{\rm dist}(x, Z(f))^c}{q^{c(s+1)}}.
	\end{equation*}
	Thus, H1 is proved for $x\in R^m$ with $\abs{f(x)}\leq q^{-N} $. For $x$ in $R^m$ with $\abs{f(x)}> q^{-N},$ we have $q^{-N}< \abs{f(x)} \leq 1$, while ${\rm dist}(x,Z(f)\cap R^m)\leq 1$ since $\norm{x-y}\leq 1$ for $x,y\in R^m$. Hence $\abs{f(x)}\geq q^{-N}{\rm dist}(x,Z(f)\cap R^m) \geq q^{-N}\textrm{dist}(x,Z(f)\cap R^m)^c\geq q^{-N}\textrm{dist}(x,Z(f))^c$. If $C=\min(q^{-N},q^{-(s+1)c})$, then we have H1 with $\alpha=c$.
\end{proof}
\begin{rem}\label{rem:4.5}
	That the local version of {\L}ojasiewicz inequality comes out of \cite{greenberg1966rational} has been observed in \cite{bollaerts1990estimate}; we give this proof since it includes the case when $k$ has characteristic $>0.$ {Greenberg's result is applicable here because $R$ is then complete ($k^*=k$ in his notation).}
\end{rem}
\subsection{Proof of H2}
\begin{lem}
	If {\rm H2} is true for a $k$-vector space $V$, then it is also true for any subspace $W$ of $V$. In particular, for a central division algebra, $D_r$ over $k$, of dimension $r^2 \geq \dim_k V $, it is  enough to prove {\rm H2} for $D_r$.
\end{lem}
\begin{proof} 
	Let $W\subseteq V$ be a subspace, and $U\subseteq V$ such that $V=W\oplus U \simeq W\times U $. Let $f$ be a polynomial on $W$. Define the polynomial $g$ on $V$ by $g(w+u)\colonequals f(w).$ For $w\in W, u\in U,$ we take $\norm{w+u}  = \text{max}(\norm{u}, \lVert w \rVert)$; because $U$ and $W$ are complementary, this is non-archimedean. Clearly $Z(g)=Z(f)\times U$. 
	
	Suppose $Z(f)=\emptyset$. Then $Z(g)=\emptyset$. Since H2 is true for $V$ and $W\subset V$, there exist constants $C>0, \beta \geq 0$ such that 
	$\abs{f(w)}\geq {C}{\norm{w}^{-\beta}}$ for $w\in W,\norm{w}\geq 1.$
	We may therefore assume that $Z(f)\neq \emptyset$, so $Z(g)\neq \emptyset.$
	
	Then, $ \abs{g(x)}  \geq C\ {{\rm dist} (x, Z(g))^{\alpha}}{\norm{x}^{-\beta}}$ for  $x \in V$, $ \norm{x} \geq 1$ where $C>0,$ $\alpha,\beta\geq0$ are constants. If $x= w\in W$, ${\rm dist}(w, Z(g))  ={\rm dist}(w,Z(f))$.	
\end{proof}

Now we prove H2 for $D_r$. Our proof is inspired by H\"{o}rmander's \cite{hormander1958division}. It replaces the inversion in his proof by the involution $x\mapsto x^{-1}$ of $D_r^{\times}\colonequals D_r\setminus \{0\} .$

For a division algebra $D_r$ of dimension $r^2$, central over $k$, let us recall $\nu \colonequals\nu_r\colon D_r \to k$ of proposition \ref{prop:3.1}, and note that it has the following property: if  $k'$ is any field containing $k$ such that there exists an isomorphism $F\colon k'\otimes_k D_r\stackrel{\sim}{\longrightarrow} M_r(k')=M_r $ where $M_r$ is the algebra of $r\times r$ matrices over $k'$, then $\nu(a)= \det F(a)$ for $a\in D_r$ (\cite{weil2013basic}, Proposition 6, p. 168).


\begin{lem}\label{lem:4.8}
	For any polynomial function $f\colon D \rightarrow k$ of degree $d$, $f$ not necessarily homogeneous, let $f^*(x)\colonequals f(x^{-1})\nu (x)^{ d}$ for $x\neq 0$; then $f^*(x)$ extends uniquely to a polynomial function $D_r\rightarrow k$. Moreover, for non-zero $x$,  $x \in Z(f)$ if and only if $x^{-1}\in Z(f^*)$.
\end{lem}
\begin{proof}
	Uniqueness is obvious. To prove that $f^*$ has a polynomial extension it suffices to prove it for $k'\otimes_k D_r$, where $k'$ is a separable extension of $k$ such that $k'\otimes_k D_r\simeq M_r(k')$. The required result is compatible with addition and multiplication of the $f's$ so that it is enough to verify it for $f=1$ (obvious) and $f=a_{ij}$, a matrix entry; then $f^*=a^{ij}\det =A_{ij}$, the corresponding cofactor. The last statement of the lemma is obvious
\end{proof}

\begin{rem}
	From now on we use the norm $\norm{x}=\abs{\nu(x)}^{1/r}$ for $D_r$, $r\geq 2$.
\end{rem}

\begin{lem}
	If $x,y,x-y$ are all nonzero, then $\lVert x-y  \rVert = \lVert x^{-1}- y^{-1} \rVert \lVert x \rVert \lVert y\rVert$
\end{lem} 

\begin{proof}
	Use $ y-x = x\left( x^{-1} - y^{-1} \right) y$ and the multiplicativity of $\norm{\cdot}$.
\end{proof}
The next two lemmas are auxiliary before we prove H2 for $D_r$.
\begin{lem} \label{lem:4.13}
	If $Z(f)$ is nonempty, there exists a constant $A\geq 1$ such that  ${\rm dist}(x,Z(f))\leq A\norm{x}$ for all $x$ with $\norm{x}\geq 1.$
\end{lem}

\begin{proof}
	Choose $z_0\in Z(f).$ Then ${\rm dist }(x,Z(f))\leq \norm{x-z_0}\leq \max\left( \norm{x},\norm{z_0}\right).$ If $\norm{x}\geq \norm{z_0}$, then ${\rm dist}(x,Z(f))\leq \norm{x}$ and we can take $A=1.$ If $ \norm{x} < \norm{z_0}$ then $\norm{x-z_0}=\norm{z_0}\leq \norm{z_0}\norm{x} $ for $\norm{x}\geq 1$; and as $\norm{z_0}\geq 1$, the lemma is proved if we take $A=1+\norm{z_0}.$
\end{proof}


%

\begin{lem}\label{lem:4.15}
	Suppose $Z(f)$ contains a nonzero element. Then there exists a constant $C>0$ such that 
	\begin{equation}\label{eq:4.4}
	  {\rm dist} (x^{-1}, Z(f^*))\geq C\ \frac{ {\rm dist}(x,Z(f)) }{\norm{x}^2} \qquad(\norm{x}\geq 1).
	\end{equation}
\end{lem}

\begin{proof} 
	First assume $0\notin Z(f^*)$. Then $Z(f^*)=Z(f^*)\setminus \{0\}\neq \emptyset.$ Then, with $\norm{x}\geq 1$,
	${\rm dist} \left(x^{-1} ,Z(f^*)\setminus \{0\}\right)$ $ =\inf_{0\neq z\in Z(f^*)} \norm{x^{-1}-z}$ 
	$= \inf_{0\neq y\in Z(f)} \norm{x^{-1}-y^{-1}}$
	$= \inf_{0\neq y\in Z(f)}E,$ where $E\colonequals {\norm{x-y}}\norm{x}^{-1}\norm{y}^{-1}.$
	
	We consider cases: (a) $\norm{y} > \norm{x}$ and (b) $ \norm{y}\leq \norm{x}$. 
	In case (a) $\norm{x-y}=\norm{y}$ so that $E=\norm{x}^{-1} = \norm{x}\norm{x}^{-2}\geq A^{-1} \, {\rm dist} (x,Z(f)){\norm{x}^{-2}}$ where $A\geq 1$ is as in lemma \ref{lem:4.13}. In case (b) 
	$E\geq {\norm{x-y}}\norm{x}^{-2}$ so that $\inf E\geq  {\rm dist} (x,Z(f))\norm{x}^{-2}.$
	These give (\ref{eq:4.4}) with $C=1/A$.
	
	If $0\in Z(f^*)$, then ${\rm dist}(x^{-1},Z(f^*)) = \min\left( {\rm dist}(x^{-1}, Z(f^*)\setminus \{0\} ), \norm{x^{-1}} \right).$
	Now $\norm{x}^{-1} = \norm{x} \norm{x}^{-2} \geq C\norm{x}^{-2}{\rm dist} (x, Z(f))$ by lemma \ref{lem:4.13} where $C=1/A$, while ${\rm dist} (x^{-1}, Z(f^*) \setminus \{0\}) \geq C\norm{x}^{-2} {\rm dist} (x, Z(f)),$ by above.
\end{proof}

\begin{proof}[Proof of {\rm H2} for $D_r$]
	We consider two cases: (a) $Z(f)=\emptyset$, (b) $Z(f)\neq \emptyset$. 
	
	{\sl Case (a):} Then $Z(f^*)=\emptyset$ or $\{0\} .$ If $Z(f^*)=\emptyset,$ then there exists a constant $C>0$ such that $\abs{f^*(x)}\geq C>0$ with $\norm{x} \leq 1$. So, $\abs{f^*(y)}=\abs{f(y^{-1})}\norm{y}^{r d} \geq C >0$ for $0<\norm{y}\leq 1,$
	which becomes $\abs{f(x)} \geq C \norm{x}^{r d} \geq C> 0$ for $\norm{x}\geq 1.$
	
	If $Z(f^*)=\{0\}$, then ${\rm dist }(z,Z(f^*)) =\norm{z}$, and $\abs{f^*(y)}\geq C\norm{y}^{\beta}$ with $0<\norm{y}\leq 1$ for constants $C>0$, $\beta \geq 0$ by theorem \ref{thm:H1}. Then
	$\abs{f(y^{-1})}\norm{y}^{r d} \geq C\norm{y}^\beta$ with $\norm{y}\leq 1$ or
	$\abs{f(x)}\geq C\ {\norm{x}^{r d}} {\norm{x}^{-\beta} } \geq {C}\norm{x}^{-\beta}$ with $\norm{x}\geq 1.$
	
	{\sl Case (b):} $Z(f)$ is now nonempty, and hence either $Z(f)=\{0\}$ or $Z(f)$ contains a nonzero element. If $Z(f)=\{0\},$ then  $Z(f^*)=\emptyset$ or $\{0\}$. This comes under case (a), above, and we have $\abs{f(x)} \geq {C}\norm{x}^{-\beta}$ with $\norm{x}\geq 1$
	which gives (a) of H2. 
	
	Suppose $Z(f)$ contains a nonzero element. By H1, there exists constants $C_1>0$, $\alpha\geq0$ such that $\abs{f^*(x^{-1})}\geq C_1\ {\rm dist}(x^{-1}, Z(f^*) )^\alpha$ with $\norm{x}\geq 1.$ So by lemma \ref{lem:4.15}, for $C_2=C_1\,C^\alpha$, 
	$\abs{f(x)} \geq C_2 \,{ {\rm dist} (x,Z(f))^\alpha }{\norm{x}^{-2\alpha} }$ for $\norm{x}\geq 1$,
	proving (b) of H2.
\end{proof}

\subsection{Criterion for a polynomial not to be rapidly decreasing on a set $S$}
 In \cite{igusa1978lectures} Igusa and Raghavan develop what is essentially a criterion for a polynomial on an real vector space \emph{not to be rapidly decreasing} on a set of vectors of norm $\geq1$. In this section we generalize that method to all local fields, introducing several polynomials in the criterion.
\begin{lem}\label{lem:4.15a}
	Let ${\bf f}\colon V\to k^r$ be a polynomial map and $d$ the maximum of the degrees of its components. Then there exists a constant $C>0$ such that for all $x,y\in V$ with $\norm{x}\geq1 $,
	\[
		\norm{{\bf f} (x)- {\bf f}(y)} \leq C\norm{x}^{d-1} \max_{0\leq r\leq d}\left( \norm{x-y}^r\right).
	\] 
\end{lem}

\begin{proof}
	It is enough to prove this for $r=1$, ${\bf f}=f$. The estimate is compatible with addition in $f$ and so we may assume $f$ to be a monomial of degree $d$ in some coordinate system on $V$. Assume the result for all monomials of degree $d-1$. Then $f=x_ig$, where $g$ is a monomial of degree $d-1$. We have
	\[
		x_ig(x)-y_i g(y) = x_i(g(x)-g(y)) +(x_i-y_i)(g(y)-g(x))+(x_i-y_i)g(x)
	\]
	and the estimate is obvious for each of the three terms.
\end{proof}

\begin{prop} \label{prop:4.16}
	Let $S\subseteq V$ be a set with $\norm{x}\geq  1 $ for all $x \in S$. Let $g$ be polynomial on $V$.  If $Z(g)= \emptyset$, we have 
	\[
		\abs{g(x)}\geq \frac{C}{\norm{x}^{\gamma} }\qquad (\norm{x}\geq 1) 
	\]
	for some $C>0$, $\gamma \geq 0$. Suppose $Z(g)\neq \emptyset $ and
	suppose that there exist polynomials $f_i\colon V\rightarrow k$, $i=1,\ldots r$, and a constant $b >0$ such that $\max{\abs{f_i(x) -f_i(y)} }\geq b >0 $ for all $x \in S$, $y\in Z(g)$. 
	Then there exist constants $C>0$ and $\gamma \geq0$ such that 
	\begin{equation}
	\lvert g(x) \rvert \geq \frac{C}{\lVert x \rVert^\gamma } \qquad (x\in S).
	\end{equation}
\end{prop}

\begin{proof} 
	The first statement is a) of H2. We now assume $Z(g)\neq \emptyset$. We identify $V\simeq k^m$, and work in coordinates. Set $d\colonequals\max_{i}(\deg(f_i))$. In what follows, $C_1,\ C_2,\ldots,$ are constants $>0$. 
	
	For all $x\in S$ and $y\in Z(g)$, by lemma \ref{lem:4.15a} for some constant $C>0$, we have $0<b\leq \max_{1\leq i\leq r}\abs{f_i(x)-f_i(y)}\leq C\norm{x}^{d-1}\max_{1\leq r\leq r}\norm{x-y}^r$
	for all $x\in S$ , $y\in Z(g)$. 

	Choose $y \in Z(g)$ such that $ \lVert x - y\rVert = \text{dist} (x, Z(g)).$
	Then for all $x\in S$, we have $$0<b \leq C_1 \lVert x \rVert^{d-1} \max_{1\leq r \leq d}\left( \text{dist}(x, Z(g))^r \right).$$
	
	We consider two cases:
		(a) $\text{dist}(x,Z(g)) \leq 1,$ so the maximum above is $\text{dist}(x,Z(g))$, and
		(b) $\text{dist}(x,Z(g)) > 1,$ so the maximum is $\text{dist}(x,Z(g))^d$.
	
	By H2 there exist constants $C_2> 0,$ $\alpha, \beta \geq 0$ such that
	$\abs{g(x)} \geq$\\ $ C_2\,{\text{dist}(x,Z(g))^\alpha }\norm{x}^{-\beta}, $
	so 
	$ \text{dist}(x,Z(g)) \leq  C_3\, \lvert g(x)\rvert^{1/\alpha} \lVert x \rVert^{\beta/\alpha} .$
	In case (a), $0< b\leq C_3\abs{g(x)}^{1/\alpha}\norm{x}^{\frac{\beta}{\alpha}+(d-1) }$, and in case (b), $0< b\leq C_4\abs{g(x)}^{d/\alpha}\norm{x}^{\frac{d \beta}{\alpha}+(d-1) }$.
	So in both cases, with $\delta = \frac{d\beta}{\alpha}+(d-1)$, one has $$0<b\leq C_5\, \lVert x \rVert^{\delta} \max(\lvert g(x) \rvert^{1/\alpha} , \lvert g(x) \rvert^{d/\alpha} ).$$
	Hence, $\max \left(\abs{g(x)} ,\abs{g(x)}^d\right) \geq {C_6}\,{\norm{x}^{-\delta \alpha}},$
	giving in all cases $\abs{g(x)}\geq C_7\,{\norm{x}^{-\delta \alpha} }$ with $x\in S.$
\end{proof}

\subsection{Lower bounds of $\norm{\nabla_r f}$ on stably non-critical level sets}

Let $V$ and ${\bf f}=(f_1,\ldots f_r)\colon V\to k^r$ ($r\leq m = \dim_k V$) be as usual. Let $C({\bf f})$ be the critical set of ${\bf f}$, and $CV({\bf f})={\bf f}(C({\bf f}))$ have their usual meanings.
Write $W=CV({\bf f})$. We assume that the closure $\overline{W}$, in the $k$-topology of $k^r$, of $W$ is a proper subset of  $k^r$. Our assumption is equivalent to assuming that stably non-critical values of ${\bf f}$ exist, which is true in characteristic zero (see \S 2.2). Let $L_{{\bf c}}$, $\nabla_r {\bf f}$, and $\partial_J {\bf f}$ be defined as in \S 2.1.

If $\omega \subset k^r\setminus \overline{W}$ is a compact set, then there exists $b>0$ such that $\norm{u-v}\geq b >0$ for $u\in \omega, v\in \overline{W} $.
This means $\max_i \abs{f_i(x) -f_i(y) } \geq b > 0,$ with ${\bf c}\in \omega, x\in L_{\bf c},\; y\in C({\bf f}).$

\begin{prop}
	Let $\omega \subset k^r$ be an open set whose closure consists entirely of non critical values of ${\bf f}=(f_1,\ldots, f_r)$. For ${\bf c}\in \omega$, let $L_{\bf c}$ be defined as above. Then there exist constants, $C, \gamma >0 $ such that 
	\begin{equation}
	\norm{\nabla_r {\bf f}(x)}\geq \frac{C}{\norm{x}^\gamma } \qquad (x\in L_{\bf c}, {\bf c}\in \omega, \norm{x}\geq 1)
	\end{equation}
\end{prop}

\begin{proof}	
	We write $(y_J)$ for the coordinates on $k^{\binom{m}{r}}$ and select a definite homogeneous form $\nu$, which is positive definite of degree 2 if $k$ archimedean, and of degree $R$ on $k^{\binom{m}{r}}$ where $R$ is any integer $\geq 2$ such that $R^2\geq \binom{m}{r},$ with the property that $\abs{\nu(y)}^{1/R}$ is a norm on $k^{\binom{m}{r}},$ if $k$ is non-archimedean. Then $\nu(\nabla_r {\bf f}(x)) =0$ if and only if $\nabla_r {\bf f}(x)=0 $, i.e., if and only if $x$ is a critical point of ${\bf f}.$ Let $g(x)=\nu(\nabla_r {\bf f}(x)).$ Then $Z(g)$ is the set of critical points of ${\bf f}$. Suppose first that $Z(g)\neq \emptyset.$ Now there exists $b>0$ such that 
	$$\norm{u-v}=\max_{1\leq i\leq r} \abs{u_i - v_i} \geq b>0 \qquad (u\in \omega, v\in \overline{W})$$
	Hence, as ${\bf f}(x)\in \omega$ for $x\in L_{\bf c}$ $(c\in \omega)$, and ${\bf f}(y)\in \overline{W}$ for $y\in Z(g)$, $\norm{{\bf f}(x)-{\bf f}(y)}\geq b> 0.$
	So by proposition \ref{prop:4.16} there exist constants $C>0,$ $\delta \geq 0$ such that 
	\[
	  \abs{\nu(\nabla_r {\bf f}(x))}=\abs{g(x)}\geq \frac{C}{\norm{x}^\delta}\qquad (x\in L_{\bf c},\ {\bf c}\in \omega ,\norm{x}\geq 1).
	\]
	But $\nu$ is homogeneous of degree $d$ ($d=2$ for archimedean and $R$ for non-archimedean $k$) and definite. So there exist constants $C_1,C_2>0$ such that $C_1\norm{\nabla_r {\bf f}(x)}^d\leq \abs{\nu(\nabla_r {\bf f}(x))} = \abs{g(x)} \leq {C_2}{\norm{\nabla_r {\bf f}(x)}^d}.$
	So for suitable $C>0$, $\gamma\geq 0$, we have $\norm{\nabla_r {\bf f}(x)} \geq {C}\norm{x}^{-\gamma}.$
	The case $Z(g)=\emptyset$ is taken care of by the first statement of proposition \ref{prop:4.16}.
\end{proof}
\begin{rem}\label{rem:4.15}
	We cannot make $\gamma=0$ in all cases. For instance, let $\text{char.}\,k=0$ and $r=1$, $f(x,y,z)=x^2 z^2+y^3z$ and $c=-1$.  Consider $x_n = n, \ z_n = \frac{1}{n}, \ y_n = -(2n)^{1/3}.$
	Then $F(x_n,y_n,z_n)= 1-2=-1,$ $\frac{\partial F}{\partial X} (x_n,y_n, z_n)  = 2x_n z^2_n \to 0$, and 
	$\frac{\partial F}{\partial Y} (x_n,y_n, z_n) = 3y_n^2 z_n \to 0,$
	$\frac{\partial F}{\partial Z}(x_n, y_n, z_n) = 2x^2_n z_n +y^3_n= 2n-2n=0.$ But $\norm{(x_n,y_n,z_n)}=n$, $\norm{\nabla f (x_n,y_n,z_m)} \sim \text{Const}\cdot \frac{1}{n^{1/3}}$. So $\gamma \geq 1/3$. We do not know the minimal value of $\gamma$.
\end{rem}
\section{Proof of temperedness of canonical measures on stably non-critical level sets}
\subsection{Consequences of Krasner's lemma.}The well-known lemma of Krasner \cite{artin67} has an important consequence (corollary \ref{cor:5.3}). Let $k$ be a local field of arbitrary characteristic and $K$ its algebraic closure. The following lemma must be well-known, but we prove it in this form.
\begin{lem}
	We can find a countable family $\{k_n\}$ of finite extensions of $k$ with the property that any finite extension of $k$  is contained in one of the $k_n$. In particular $K=\bigcup_n k_n$.
\end{lem}
\begin{proof}
	We first work with separable extensions of fixed degree $n$ over $k$. Let $S_n$ be the set of monic, irreducible and separable elements of $k[X]$ of degree $n$. Then it follows from Krasner's lemma that if $f\in S_n$, there is an $\varepsilon=\varepsilon (f)>0$ with the following property: if $g$ is monic and $||f-g||<\varepsilon$, then $g\in S_n$ and $K(f)=K(g)$ in $K$, where $K(h)$ denotes the splitting field of $h$. Since $S_n$ is a separable metric space, it follows that there are at most a countable number of these splitting fields, and any separable extension of degree $n$ over $k$ is contained in one of these. Let us enumerate these splitting fields as $\{k_{nj}\} (j=1,2,\dots)$. If $k$ has characteristic $0$ we are already finished. Suppose $k$ has characteristic $p>0$. Let $F(x\mapsto x^p)$ be the Frobenius automorphism of $K$. Define the extension $k_{njr}=F^{-r}(k_{nj})$ for $r=1,2,\dots$, which are clearly finite over $k$. Clearly, any finite extension of $k$ of finite degree is contained in one of the $k_{njr}$.
\end{proof}

\begin{rem}
	If $k$ has characteristic $0$, then there are only a finite number of extensions of fixed degree $n$. But in prime characteristic this is not true: the field $k=F_2[[X]][X^{-1}]$ of Laurent series in $X$ with $F_2$ a finite field of characteristic $2$ has a countably infinite number of separable quadratic extensions. Indeed, the extensions defined by $T^2-T-c=0$ are distinct for infinitely many values of $c$.
\end{rem}

\begin{cor}\label{cor:5.3}
	If $M$ is an affine subvariety of some  $A_K^n$ and $M(k')$ is countable for all finite extensions $k'$ of $k$, then $M$ is finite.
\end{cor}	

\begin{proof}
	By Lemma 5.1, $M(K)=\bigcup_{k'} M(k')$ is countable, hence finite.
\end{proof}

\subsection{A consequence of the refined Bezout's theorem} The refinement of Bezout's theorem due to Fulton and MacPherson (\cite{FM1977}, \cite{Vog1984}) is the statement that if $Z_i \ (1\le i\le r)$ are $r \ (r \ge 2)$ pure dimensional varieties in $\mathbb{P}_K^m$, then the number of irreducible components of $\bigcap _i Z_i$ is bounded by the Bezout number $\prod_i{\rm deg}(Z_i)$. It has the following simple consequence.

\begin{lem}\label{lem:5.4}
	\it Let $U$ be a nonempty Zariski open subset of $\mathbb{A}_K^r$ so that $U\subset \mathbb{A}_K^r\subset \mathbb{P}_K^r$. Let $h_i \ (i=1,2, \dots ,r)$ be polynomials on $A_K^r$ with $\deg h_i \equalscolon d_i$, and let $Z_i$ be the zero locus of $h_i$. Let $Z_i^\times=Z_i\cap U$ and  $\overline {Z_i}$ the closure of $Z_i$ in $\mathbb{P}_K^r$. If $\bigcap_i Z_i^\times=F$ is non-empty and finite, then $F$ has at most $D\colonequals\prod _id_i$ elements.
\end{lem}
\begin{proof}
	Since $\mathbb{A}^r_K$ is Zariski dense in $\mathbb{P}_K$ we have $\overline Z_i\bigcap A_K^r=Z_i$; moreover, $\overline{Z}_i$ is of pure degree $d_i$. Let $W_0$ be an irreducible component of $W\colonequals\bigcap \overline Z_i$ that meets $U$. Since $W_0$ is irreducible and $W_0\cap U$ is nonempty open in $W_0$, it is dense in $W_0$. Let $w \in W_0\bigcap U$. Then $w$ is in each of the $\overline Z_i\bigcap U$ and so $w\in F$. So $W_0\cap U$ is finite and contained in $F$. Since $W_0\bigcap U$ is dense in $W_0$, it follows that $W_0\bigcap U$ must consist of a single element of $F$ and $W_0$ itself consists of that point. Moreover all points of $F$ are accounted for in this manner as $F$ is contained in the union of irreducible components of $W$ which meet $U$. Hence the cardinality of $F$ is at most the number of irreducible components of $W$, which is at most $D$.
\end{proof}

\subsection{The maps $\pi_J$ and a universal bound for the cardinality of their fibers.}
Let $V\simeq k^m$ so that ${\bf f}=\left(f_1,\ldots, f_r\right)$ with $f_j\in k[x_1,\ldots,x_m]$. Assume that ${\bf c}$ is a non-critical value of ${\bf f}$ so that $L_{\bf c}$ has no singularities. Fix $J\subset \underline{m}\colonequals\{1,\ldots,m\}$, and let $\pi_J\colon k^m\to k^{m-r}$ map $(x_1,\ldots,x_m)$ to $(y_1,\ldots,y_{m-r})$ where $\{y_j\}_{j=1}^{m-r}=\{x_i \,|\, i\in \underline{m}\setminus J \}$. We wish to prove that the map $\pi_J$ restricted to $L_{\bf c}$ has has fibers of cardinality $\leq D\colonequals d_1\cdots d_r$, where $d_i\colonequals\deg(f_i)$. Without loss of generality assume $J=\{1,\ldots,r\},$ so that $\pi_J\colon \left(x_1,\ldots,x_m\right)\mapsto \left(x_{r+1},\ldots,x_m\right).$ Write $x=\left(x_1,\ldots,x_m\right)$ and $y=\left(x_{r+1},\ldots,x_m\right)$. Define $z$ so that $x=(z,y)$.

We regard $L_{\bf c}$ as an affine variety and $L_{{\bf c},J}$ as an affine open subvariety. For any $k'$ with $k\subset k'\subset K$ we have the respective sets of $k'$-points, $L_{\bf c}(k')$ and $L_{{\bf c},J}(k')$. Denote the restriction of $\pi_{J}$ to $L_{{\bf c},J}$ by $\overline{\pi}_J$.
\begin{prop}
	Let $D=\prod_{1\leq i\leq r} d_i$. Then the fibers of $\overline{\pi}_J$ are all of cardinality $\leq D$.
\end{prop}
\begin{proof}
	Note that $d\overline{\pi}_J$ is an isomorphism on $L_{{\bf c},J}(k)$. Hence $U_J(k)\colonequals\overline{\pi}_J\left(L_{{\bf c},J}(k)\right)$ is open in $k^{m-r}$ and $\overline{\pi}_J$ is a local analytic isomorphism of $L_{{\bf c},J}(k)$ onto $U_J(k)$. For any field $k'$ between $k$ and $K$, we write again $\overline{\pi}_J$ for the map $L_{{\bf c},J}(k')\to k'^{\,m-r}$, and $U_J(k')$ for its image. If $k'$ is a \emph{finite} extension of $k$, then $k'$ is again a local field; exactly as for $k$, we have $d\overline{\pi}_{J}\colon L_{{\bf c},J}(k')\to U_J(k')$ is an analytic isomorphism. For any $k'$, $k\subset k'\subset K$ with $k'/k$ finite, $U_J(k')$ is open in $k'^{\, m-r}$ and the fibers of $\overline{\pi}_J$ on $L_{{\bf c},j}(k')$ are discrete and at most countable. If we then fix $y\in U_J(k)$, and write $W_y$ for the affine variety $\overline{\pi}_J^{-1}(y)$, then $W_y(k')$ is at most countable for all finite extensions $k'/k$. Hence, by corollary \ref{cor:5.3}, $W_y(K)$ is finite. Let $F\colonequals W_y(K)$.
	
	On the other hand, $\pi_J^{-1}(y)(K)=K^r\times\{y\}\simeq K^r.$ Let $h_i(z)\colonequals f_i(z,y)-c_i.$ Then $h_i$ is a polynomial on $K^r$ of degree $\leq d_i$. Moreover, since $\overline{\pi}_J^{-1}(y)(k)$ is nonempty, $\frac{\partial(h_1,\ldots,h_r)}{\partial(x_1,\ldots,x_r)}=\partial_J(z,y)$ is not identically zero on $K^r$. Thus, $ \left\{ z | \partial_J(z,y)\neq 0 \right\}$ is a non-empty affine open $U_1$ in $K^r$. Moreover, $F=\bigcap_{1\leq i \leq r}Z(h_i)^\times$ where $Z(h_i)^\times\colonequals Z(h_i)\cap U_1$. So Lemma \ref{lem:5.4} applies and proves that $\#{F}\leq D$.
\end{proof}

\begin{lem}
	Let $\partial _J$ be as in \S 2.2. Then if $\omega _{m-r}$ is the exterior form corresponding to the Haar measure on $k^{m-r}$, the exterior form
	$$
	\rho_{\bf c}\colonequals{1\over \partial_J(x)} \overline{\pi}_J^\ast (\omega_{m-r})
	$$
	on $L_{{\bf c}, J}$ has the property that $|\rho_{\bf c}|$ generates the measure $\mu_{\bf c}\colonequals\mu_{{\bf f, c}}$. In particular, if $\lambda $  is the Haar measure on $k^{m-r}$ and $\nu$ is the measure generated by $\abs{\overline{\pi}_J^\ast (\omega_{m-r})}$, then $\overline{\pi}_J$ takes $\nu$ to $\lambda$ in small open neighborhoods of each point of $L_{{\bf c}, J}(k)$, and $d\mu_{\bf c}=|\partial_J(x)|^{-1} d\nu$.
\end{lem}
\begin{proof}
	This is clear from (\ref{eqn:measure}).
\end{proof}	

\subsection{Proof of Theorem \ref{thm:1.1}.} 
This follows from three things: the lower bounds on $\norm{\nabla_r }$ when $\bf c$ is a stably non-critical value of $\bf f$, the relationship between $\lambda, \nu, \mu_{\bf f, c}$, and the temperedness of $\lambda$. The simple measure-theoretic lemma below explains this.
Let $R, S$ be locally compact metric  spaces which are second countable, with Borel measures $r,s$ respectively on them, and $\pi (R\rightarrow S)$ a continuous surjective map which is a local homeomorphism, and takes $r$ to $s$ in a small neighborhood of each point of $R$: this means that for each $x\in R$ there are open sets $M_x, N_{\pi (x)}$ containing $x$ and $\pi(x)$ respectively, such that $\pi$ is a homeomorphism of $M_x$ with $N_{\pi(x)}$ and takes $r$ to $s$. 
\begin{lem} 
	If there is a natural number $d$ such that all fibers of $\pi$ have cardinality  at most $d$, then for each Borel set $E\subset R$, $\pi(E)$ is a Borel set in $S$, and we have 
	$$
	r(E)\le d{\cdot}s(\pi(E)).
	$$
	Moreover if $f\ge 0$ is a continuous function on $R$ and $t$ is the Borel measure on $R$ defined by $dt=fdr$, then for any Borel set $E\subset R$ we have
	$$
	t(E)\le \sup_E|f|{\cdot}d{\cdot}s(\pi (E)).
	$$
\end{lem}

\begin{proof} 
	The second inequality follows trivially from the first, so that we need only prove the first. We use induction on $d$. For $d=1$, $\pi$ is a continuous bijection of $R$ with $S$ ; being a local homeomorphism, it is then a global homeomorphism. It is easy to see that it takes $r$ to $s$ globally, and so the results are trivial. 
	Let $d>1$, assume the results for $d-1$, and suppose that there are points of $S$ the fibers over which have cardinality exactly $d$. Let $S_d$ be the set of such points in $S$. Now, if the fiber above a point has $e$ elements, the fibers of neighboring points have cardinality $\ge e$, and so $S_d$ is open in $S$. Let $R_d=\pi^{-1}(S_d)$. Then $\pi \colon R_d\rightarrow S_d$ is a $d$-sheeted covering map. If $x\in R_d$, we can find an open set $M$ containing $\pi(x)$ such that $N\colonequals\pi^{-1}(M)=\bigsqcup_{1\le j \le d}N_j$ where $\pi (N_j\rightarrow M)$ is a homeomorphism taking $r$ to $s$. If $E\subset N$ is a Borel set, then $E=\bigsqcup_jE \cap N_j$, so that $\pi(E)=\bigcup_j\pi(E\cap N_j)$ is Borel as $\pi$ is a homeomorphism on each $N_j$. Moreover,
	$$
	r(E)=\sum_j r(E\cap N_j)=\sum_js(\pi(E\cap N_j)\le d{\cdot}s(\pi(E)).
	$$
	These two properties are true with any Borel $M'\subset M$ and $N'=\pi'(M')$ replacing $M, N$ respectively. Write now $S_d=\bigcup_n M_n$ where the $M_n$ are open and have the properties described above for $M$. Then $S_d=\bigsqcup_nM'_n$ where $M'_n\subset M_n$, so that $R_d=\bigsqcup_n\pi^{-1}(M'_n)$. The two properties above are valid for any Borel set contained in any $\pi^{-1}(M'_n)$, hence they follow for any Borel set $E\subset R_d$.
	Write $S'=S\setminus S_d, R'=\pi^{-1}(S')=R\setminus R_d$. Then $(R',S', \pi )$ inherit the properties of $(R,S,\pi)$ with $d-1$ instead of $d$. The result is valid for $(R',S',\pi)$ and hence for $(R,S,\pi)$, as is easily seen.
	
	We are now ready to prove Theorem \ref{thm:1.1}. Assume that ${\bf c}$ is a stably non-critical value of ${\bf f}$. For simplicity of notation we will suppress mentioning ${\bf c}$, because all of our estimates are locally uniform in ${\bf c}$. On $L_{\bf c}=L$ we have the estimate
	$$
	||\nabla_r(x)||=\max _J |\partial _J(x)|> {C\over ||x||^\gamma}\qquad (||x||\ge 1)
	$$
	where $C>0, \gamma \ge 0$ are constants that remain the same when ${\bf c}$ is varied in a small neighborhood of ${\bf c}$. Let us write $L^+$ for the subset of $L$ where $||x||>1$. Now, at each point $x\in L^+$ some $|\partial_J(x)|$ equals $||\nabla_r(x)||$. Hence if we write
	$$
	M_J=\{x\in L^+\ |\ |\partial_J(x)|>C||x||^{-\gamma}\}
	$$  
	then
	$$
	L^+=\bigcup_JM_J.
	$$
	The map $\overline{\pi}_J$ is open on $M_J$ onto its image $W_J$ and is a local analytic isomorphism. Moreover, if $\lambda, \nu,\mu=\mu_{\bf c}$ have the same meaning as before, we have, on $M_J$, 
	$$
	d\mu=|\partial_J(x)|^{-1} d\nu
	$$
	and hence, for any Borel set $E\subset M_J$, with $D$ as in lemma \ref{lem:5.4},
	$$
	\mu(E)\le D{\cdot}\sup_E|\partial_J(x)^{-1}|{\cdot}\lambda(\pi_J(E)).
	$$
	Remembering that $|\partial_J(x)|^{-1}< C^{-1}||x||^\gamma$ we get from this that
	$$ 
	\mu(E)\le DC^{-1}{\cdot}\sup_E||x||^\gamma{\cdot}\lambda(\pi_J(E)).
	$$	
	If we take $E=B_R\cap M_J$ where $B_R=\{x\in k^m\ |\ ||x||<R\}$, we see that $\pi_J(E)$ is a subset of the open ball of $k^{m-r}$ of radius $R$, and hence $\lambda (\pi_J(E))\le AR^{m-r}$ where $A$ is a universal constant. Hence
	$$ 
	\mu (B_R\cap M_J)\le ADC^{-1}{\cdot}R^{m-r+\gamma}.
	$$
	Since this is true for all $J$, the temperedness of $\mu$ together with the growth estimate is proved, as well as the assertion that the last estimate remains unchanged if ${\bf c}$ varies in a small neighborhood of its original value. This finishes the proof of Theorem \ref{thm:1.1}.
\end{proof}

\section{Invariant measures on regular adjoint orbits of a semi simple Lie algebra.} As an application of our theorem \ref{thm:1.1} we shall prove that the invariant measures on regular semi simple orbits of a semi simple Lie algebra $\mathfrak g\colonequals \mathfrak{g}_K$ over a local field $k$ of {\it characteristic $0$} are tempered, at least when the algebra $J(k)$ of polynomial functions $\mathfrak g\to k$ invariant under the adjoint group is \emph{freely generated}. We write, for any extension $k_1$ of $k$, $J(k_1)$ a the invariant polynomial functions $\mathfrak{g}_{k_1}=\colon k_1\otimes_k \mathfrak{g}\to k_1$. Let $G(k_1)$ be the $k_1$-points of the adjoint group of $\mathfrak g$ which is defined over $k$. Let $K$ be the algebraic closure of $k$ and $\mathfrak g_K=K\otimes _k\mathfrak g$. Since $J(K)$ is freely generated by Chevalley's theorem, the field generated over $k$ by the coefficients of such a system of free generators, say $k'$, is local. Hence this assumption can be ensured by going over to $k'$. {\it We make this assumption on $k$ itself in this section.} It is satisfied if $\mathfrak g$ has a CSA split over $k$, e-g., $\mathfrak g=\mathfrak s\mathfrak l (n+1, k)$. For background material see \cite{var2}.
Let $r={\rm rank} (\mathfrak g)$. Then by assumption we can choose $g_1, \dots ,g_r \in J(k)$  freely generating $J(k)$, hence also $J(K)$ (over $K$). An element $H\in \mathfrak g_K$ is semi simple (resp. nilpotent) if ${\rm ad}\, X$ is semi simple (resp. nilpotent). A semi-simple element $H$ is called {\it regular} if its centralizer is a Cartan sub-algebra (CSA) of $\mathfrak g_K$. There is an invariant polynomial $D\in J(k)$, called the discriminant of $\mathfrak{g}$, such that if $X\in \mathfrak{g}_k$, $X$ is semi-simple and regular if and only if $D(X)\neq 0$. If $Y\in \mathfrak g$ is any element, we can write $Y=H+X$ where $H$ is semi simple and $X$ is a nilpotent in the derived algebra of the centralizer of $H$ in $\mathfrak g_K$ (which is semi simple). It is known \cite{Kos63} that the orbit of $H+X$ has $H$ in its closure, and so, for any $g\in J(K)$, we have $g(H)=g(H+X)$. If $\mathfrak h_K$ is a CSA of $\mathfrak g_K$, it is further known that the restriction map from $\mathfrak g_K$ to $\mathfrak h_K$ is an isomorphism of $J(K)$ with the algebra $J(\mathfrak h_K)$ of polynomials on $\mathfrak h_K$ invariant under the Weyl group $W_K$ of $\mathfrak h_K$. It is known that the differentials $dg_1,\dots ,dg_r$ are linearly independent at an element $Y$ of $\mathfrak g_K$ if and only if $Y$ lies in an adjoint orbit of maximal dimension, which is $\dim (\mathfrak g_K)-{\rm rank}(\mathfrak g_K)=n-r$ where $n=\dim (\mathfrak g_K)$ \cite{Kos63}. If $Y$ is  semi simple, this happens if and only if $Y$ is regular. Let $\mathfrak g_K^\prime$ be the invariant open set of regular semi simple elements. We write
$$
{\bf F}=(g_1,\dots ,g_r) \colon \mathfrak g_K\longmapsto K^r
$$
and view it as a polynomial map of $\mathfrak g_K$ into $K^r$ commuting with the action of the adjoint group. Before we apply theorem \ref{thm:1.1} to this set up, we need some preliminary discussion. Let
${\cal R}={\bf F}(\mathfrak g_K^\prime).$ The next lemma deals with the situation over $K$.
\begin{lem} 
  We have $\mathfrak g_K^\prime={\bf F}^{-1}({\cal R}).$ Moreover ${\cal R}$ is Zariski open in $K^r$, and is precisely the set of noncritical values of ${\bf F}$, so that all the noncritical values are also stably noncritical. Moreover, for any ${\bf c}\in {\cal R}$, the pre-image ${\bf F}^{-1}({\bf c})$ is an orbit under the adjoint group, consisting entirely of regular semi simple elements, hence smooth.
\end{lem}

\begin{proof}
	Since $dg_1\wedge\dots \wedge dg_r\neq 0$ everywhere on $\mathfrak{g}'_K$, the map ${\bf F}$ is smooth on $\mathfrak{g}'_K$. Hence it is an open map (\cite{gor2010}, Corollary 14.34), showing that ${\bf F}(\mathfrak{g}'_K)=\mathcal{R}$ is open in $K^r$.
	
	We shall prove that if $Y\in \mathfrak g_K$ and $X\in \mathfrak g_K^\prime$ are such that ${\bf F}(Y)={\bf F}(X)$, then $Y$ is regular semi simple, and is conjugate to $X$ under the adjoint group. Suppose $Y$ is not regular semi simple. Write $Y=Z+N$ where $Z$ is semi simple and $N$ is a nilpotent in the derived algebra of the centralizer of $Z$. The ${\bf F}(Y)={\bf F}(Z)={\bf F}(X)$. Using the action of the adjoint group separately on $X$ and $Z$ we may assume that $X, Z\in \mathfrak h_K$ where ${\mathfrak h}_K$ is a CSA, and ${\bf F}(X)={\bf F}(Z)$. Then all Weyl group invariant polynomials take the same value at $Z$ and $X$ and so $Z$ and $X$ are conjugate under the Weyl group. But as $X$ is regular, so is $Z$, hence $N=0$ or $Y$ itself is regular semi simple. So, $\mathfrak{g}'_K={\bf F}^{-1}(\mathcal{R})$. But then the above argument already shows that $Y$ and $X$ are conjugate under the adjoint group. Since the fibers of ${\bf F}$ above points of ${\cal R}$ are smooth, all points of ${\cal R}$ are stably non-critical. 
	It remains to show that there are no other non-critical values. Suppose $Y\in \mathfrak g_K$ is  such that ${\bf d}={\bf F}(Y)$ is a non-critical value where ${\bf d}\notin {\cal R}$. Then $Y\notin \mathfrak{g}'_K$. Now $Y=Z+N$ as before, where $Z$ is no longer regular (it is semi simple still). Then ${\bf F}(Z)={\bf F}(Y)$ and so $Z\in {\bf F}^{-1}({\bf d})$. But as $Z$ is semi simple but not regular, $dg_1\wedge \dots ,\wedge dg_r$ is zero at $Z$ \cite{Kos63}. Thus $Z$ is a singular point of ${\bf F}^{-1}({\bf d})$, contradicting the fact that ${\bf d}$ is non-critical. The lemma is thus completely proved. 
\end{proof}
We now come to the case where the ground field is $k$, a local field of characteristic $0$. We assume that the $g_i$ have coefficients in $k$. Fix a regular semi simple element $H_0$ in $\mathfrak g_k$. Let
$$
  W(k)\colonequals W_{H_0}(k)=\{X\in \mathfrak g(k) | g_i(X)=g_i(H_0) (1\le i\le r)\}.
$$ 
\begin{thm} 
	{Assume that $J(k)$ is freely generated. Then the canonical measure on $W(k)$ is tempered, and the growth estimate {\emph{(\ref{G})}} (see \S 1) is uniform when $H$ varies in a neighborhood of $H_0$.}
\end{thm}
\begin{proof} 
	For the map ${\bf F}$ on $\mathfrak g_k$ we know that $(g_1(H_0), \dots ,g_r(H_0))$ is a stably non-critical value and so the theorem follows at once from Theorem \ref{thm:1.1}. 
\end{proof}
Although $W(K)$ is a single orbit under $G(K)$, this may no longer true over $k$.
$W(k)$ is a $k$-analytic manifold of dimension $n-r$. On the other hand the stabilizer in $G(k)$ of any point of $W(k)$ has dimension $r$ and so its orbit under $G(k)$ is an open sub-manifold of $W(k)$. If we do this at every point of $W(k)$ we obtain a decomposition of $W(k)$ into a disjoint union of $G(k)$-orbits which are open sub-manifolds of dimension $n-r$ and so all these sub-manifolds are closed also. Thus the orbit $G(k).H_0$ is an open and closed sub-manifold of $W(k)$ of dimension $n-r$. Now the canonical measure on $W(k)$ is invariant under $G(k)$ and so on the orbit $G(k).H_0$ it is a multiple of the invariant measure on the orbit. Note that the orbit being closed, the invariant measure on it is a Borel measure on $\mathfrak g_k$. Since the canonical measure is tempered on $W(k)$ by Theorem 2, it is immediate that the invariant measure on the orbit $G(k).H_0$ is also tempered. Hence we have proved the following theorem.
\begin{thm} 
	{Assume that $J(k)$ is freely generated. Then the orbits of regular semi simple elements of $\mathfrak g_k$ are closed, and the invariant measures on them are tempered.}
\end{thm}
For temperedness of invariant measures on semi simple symmetric spaces at the Lie algebra level over ${\mathbb{R}}$, see \cite{Heck1982}.

\begin{rem}
	Deligne and Ranga Rao \cite{RangaRao72} have independently shown that for any $X\in\mathfrak{g}_k$, there is an invariant measure on the adjoint orbit of $X$, and this measure extends to a Borel measure on the $k$-closure of the adjoint orbit of $X$. It is natural to ask if these are tempered in our sense when $k$ is non-rchimedean. We shall consider this question in another paper since it does not follow from the results proved here.
\end{rem}

\section{Examples}
In this section we give some examples. We consider only single polynomials ($r=1$) of degree $d \geq 3$, defined over a local field $k$ of characteristic 0. Let $f\in k[x_1,\ldots,x_m]$.

\subsection{Elementary methods when $r=1$ and $f$ is homogeneous}
For $f$ homogeneous we have Euler's theorem on homogeneous functions, which asserts that $\sum_{i} x_i \partial f/\partial x_i =d\cdot f$. Let $L_c=\left\{ x\in k^m | f(x)=c \right\}$ for $c\in k$. Then, for any critical point $x$ of $f$, we have $f(x)=0$, i.e., $L_0$ contains all the critical points. So every $c\in k\setminus \{0\}$ is a noncritical value and so is also stably noncritical. Moreover, Euler's identity for $x\in L_c$, $c\neq 0$, gives $\sum_{i} x_i\partial f/\partial x_i =dc$, so that we have
\begin{displaymath}
	\abs{d}\abs{c}=\abs{\sum_{i} x_i\frac{\partial f}{\partial x_i} }\leq C\norm{x}\norm{\nabla f(x)}\qquad (C>0)
\end{displaymath}
giving the estimate, with $A$ a constant $>0$,
\begin{displaymath}
	\norm{\nabla f(x)} \geq A\norm{x}^{-1}\qquad (\norm{x}\geq1,\;x\in L_c).
\end{displaymath}
Moreover the projection $(x_1,\ldots,x_m)\mapsto (x_1,\ldots,\hat{x}_i,\ldots x_m)$ has the property that all fibers have cardinality $\leq d$. We thus have theorem \ref{thm:1.1} with 
\begin{displaymath}
	\mu_{f,c} =O(R^{m})\qquad(R\to\infty),
\end{displaymath}
where $O$ is uniform locally around $c$. We can actually say more.
\begin{prop}\label{prop:7.1}
	Suppose ${\bf 0}$ is the only singularity in $L_0$, i.e. the projective locus of $L_0$ is smooth. Then for any compact set $W\subset k\setminus \{{\bf0}\}$, we have
	\begin{equation}\label{eq:7.1}
		\inf_{c\in W, x\in L_c, \norm{x}\geq 1}\ \norm{\nabla f(x)} > 0.
	\end{equation}
	Moreover, the measure $\mu_{f,0}$ defined on $L_0\setminus\{0\}$ is finite in open neighborhoods of ${\bf 0}$ if $m > d$, so that it extends to a Borel measure on $L_0$. Finally, for all $c\in k$,
	\begin{displaymath}
		\mu_{f,c}(B_r)=O(R^{m-1}).
	\end{displaymath}
	If $m\leq d$, there are examples where $\mu_{f,0}$ is not finite in neighborhoods of $\bf 0 $.
\end{prop}
\begin{proof}
	To prove (\ref{eq:7.1}) assume (\ref{eq:7.1}) is not true. Then we can find sequences $c_n\in W$, $x_n\in L_{c_n}$ such that $c_n\to c\in W$, $\nabla f(x_n)\to 0$ as $n\to \infty$. By passing to a subsequence and permuting the coordinates we may assume that $x_n=(x_{n1},\ldots,x_{nm})$ where $\abs{x_{n1}}\geq \abs{x_{nj}}$ $(j\geq 2)$ and $\abs{x_{n1}}\to \infty$. Now,
	\begin{displaymath}
		f(x_{n1},\ldots,x_{nm})=x^d_{n1} f(1,x_{n1}^{-1}x_{n2},\ldots, x_{n1}^{-1}x_{nm})=c_n\to c
	\end{displaymath}
	and
	\begin{displaymath}
		(\nabla f)(x_{n1},\ldots,x_{nm})=x_{n1}^{d-1}(\nabla f)(1,x_{n1}^{-1}x_{n2},\ldots, x_{n1}^{-1}x_{nm})\to 0.
	\end{displaymath}
	Now $\abs{x_{n1}^{-1}x_{nj}}\leq 1$ for $2\leq j\leq m$ and so, passing to a subsequence, we may assume that $x_{n1}^{-1}x_{nj}\to v_j$ for $j\geq 2$. Hence,
	\begin{displaymath}
		f(1,v_2,\ldots,v_m)=0\quad\text{and}\quad (\nabla f)(1,v_2,\ldots,v_m)=0
	\end{displaymath}
	showing that $(1,v_2,\ldots,v_m)\neq(0,\ldots,0)$ is a singularity of $L_0$. Then (\ref{eq:7.1}) leads to the conclusion 
	\begin{displaymath}
		\mu_{f,c}(B_R)=O(R^{m-1})\qquad(R\to\infty)
	\end{displaymath}
	locally uniformly at each $c\neq 0$.
	
	For $\mu_{f,0}$ defined on $L_0\setminus\{0\}$, one must first show that it is finite on small neighborhoods of $0$, i.e., it extends to a Borel measure on $L_0$, if $m>d$. Let $S=\left\{ u\in L_0 |\, \norm{u}=1 \right\}$. Then there exist constants $a,b>0$ such that $a\leq\norm{\nabla f(x)}\leq b$ for all $x\in S$. Hence, by homogeneity,
	\begin{displaymath}
		a\norm{x}^{d-1}\leq \norm{\nabla f(x)}\leq b\norm{x}^{d-1}\qquad(x\in L_0\setminus\{0\}).
	\end{displaymath} 
	Hence
	\begin{displaymath}
	\norm{\nabla f(x)}\geq a > 0\qquad (x\in L_0, \norm{x}\geq1).
	\end{displaymath}
	As before, this leads to $\mu_{f,0}(B_R\setminus B_1)=O(R^{m-1})$ as $R\to \infty$. Around ${\bf 0}$ we obtain the finiteness of $\mu_{f,0}$ from the estimate $b^{-1}\norm{x}^{-(d-1)}\leq \norm{\nabla f(x)}^{-1}\leq a^{-1}\norm{x}^{-(d-1)}$ and the fact that 
	\begin{displaymath}
		\int_{x\in k^{m-1},0<\norm{x}<1} \norm{x}^{-(d-1)}d^{m-1}x < \infty
	\end{displaymath}
	if $m > d$ for both $k=\mathbb{R}$ and $k$ non-archimedean.
	We shall now suppose that $f=X^4+Y^4-Z^4$. Then ${\bf 0}$ is the only critical point. The map $(x,y,z)\mapsto (x,y)$ on $L_0\cap \left\{(x,y,z)\,|\,x>0 \right\}$ is a diffeomorphism and the measure $\mu_{f,0}$ is 
	\begin{displaymath}
		\frac{1}{\abs{\partial f/\partial z}} dxdy= \frac{1}{4}\frac{dxdy}{(x^4+y^4)^{3/4}}
	\end{displaymath}
	and it is easy to verify that 
	\begin{displaymath}
		\iint_N\frac{dxdy}{(x^4+y^4)^{3/4}}=\infty
	\end{displaymath}
	for any neighborhood $N$ of $(0,0)$.
\end{proof}

\begin{rem}
	It follows from proposition \ref{prop:7.1} that to have
	\begin{equation}
		\inf_{x\in L_c,\norm{x}\geq 1} \norm{\nabla f(x)}= 0\qquad (c\neq 0) \label{eq:7.2}
	\end{equation}
	we must look for $f$ such that $L_0$ has singular points $\neq {\bf 0}$. In the next section we describe some of these examples.
\end{rem}

\subsection{Some hypersurfaces in $\mathbb{P}_k^{m-1}$ with  $[1:0:\ldots: 0]$ as an isolated singularity}
We do not try to give a ``normal form'' for such hypersurfaces; nevertheless large families of these can be described. We work in $k^{m}$, $k$ a local field of characteristic 0. Since the first coordinate axis in $k^{m}$ is chosen to be an isolated critical line (ICL), the first variable will be distinguished in what follows. Let us write $X,Y_1,\ldots,Y_{m-1}$ as the variables. Write ${\bf Y}=(Y_1\ldots,Y_{m-1})$. Let $C(\varepsilon)=\{(X,{\bf Y}) | \norm{{\bf Y}}\leq \varepsilon \abs{X}\}$

\begin{lem}\label{lem:7.3}
	Suppose $(X_n,{\bf Y_n})$ is a sequence of points in $L_c$ $(c\neq 0)$ such that they are in $C(\varepsilon)$ for some $\varepsilon <1$. Let $F(X_n, {\bf Y}_n)=c\neq 0$ and $\nabla F(X_n,{\bf Y}_n) \to 0.$
	Then if the $X$-axis is an (ICL) for $F$, we must have $X_n\to \infty$, $\frac{1}{X_n}{\bf Y}_n\to {\bf 0}$ as $n\to\infty$.
\end{lem}
\begin{proof}
	By Euler's theorem, there is no singularity on $L_c$ ($c\neq 0$). Hence $\norm{\nabla F}$ is bounded away from $0$ on each compact subset of $L_c$. Hence, item 2 above, implies $ \norm{(X_n, {\bf Y}_n)}=\abs{X_n}\to \infty.$ Then $\norm{{X_n}^{-1} {\bf Y}_n} \leq 1$ and has a limit point ${\bf \eta}$. Passing to a subsequence, if necessary, we have ${X_n}^{-1} {\bf Y}_n\to {\bf \eta}$ as $n\to \infty.$
	If $d=\deg(F)$ we have  $X_n^{d}F(1,{X_n}^{-1} {\bf Y}_n) = c$, 
	$X_n^{d-1}\partial_X F(1,{X_n}^{-1} {\bf Y}_n)\to 0,$ and  $X_n^{d-1}\partial_{Y_i} F(1,{X_n}^{-1} {\bf Y}_n)\to 0.$
	So $F(1,{\bf \eta}) =0$ and $\nabla F(1,\eta)=0,$
	while ${\bf \eta}\in C(\varepsilon).$ Hence ${\bf \eta}= {\bf 0}$ since $\varepsilon$ can be arbitrarily small.
\end{proof}
\begin{lem}
	If $(1,{\bf 0})$ is a critical point of $F$, then $F$ has the form
	\begin{displaymath}
	F=X^{d-2}p_2+ X^{d-3}p_3+\cdots + p_d
	\end{displaymath}
	where $p_r$ is a homogeneous polynomial in ${\bf Y}$ of degree $r$.
\end{lem}
\begin{proof}
	Write $F=X^{d-2}p_2+ X^{d-3}p_3+\cdots + p_d.$
	Then $p_0$ is a constant, and $F(1,{\bf 0})={\bf 0}$ gives $p_0=0.$ Then, $\frac{\partial F}{\partial Y_i}(1,{\bf 0})=0$ gives $p_1=0.$
\end{proof}
	
	From now on we let $d\geq 3$ and write
	$$F=X^{d-2}p_2 +\cdots+p_d, \quad G=p_2+\cdots +p_d.$$
	Note that $G$ is a polynomial in ${\bf Y}$, but \emph{not necessarily} homogeneous.

\begin{lem}\label{lem:ex3}
	If ${\bf 0}$ is an isolated critical point (ICP) of $G$, then the $X$-axis is an ICL of $F.$ In particular, this is so if the quadratic form $p_2$ is non-degenerate.
\end{lem}
\begin{proof}
	We must prove that if $(1,{\bf Y}_n)$ is a CP for $F$ with ${\bf Y}_n\to {\bf 0},$ then ${\bf Y}_n=0$ for $n\geq 1$. The conditions for $(1,{\bf Y}_n)$ to be a CP of $F$ are
	\begin{align*}
	F(1,{\bf Y}_n)=0, \quad \frac{\partial}{\partial X}F(1,{\bf Y}_n) =0, \quad \frac{\partial}{\partial {\bf Y}_i} F(1,{\bf Y}_n)=0 \ \text{for all } i.
	\end{align*}
	Consequently $G({\bf Y}_n)=0$ and $\frac{\partial G}{\partial {\bf Y}_i}({\bf Y}_n)=0 $ for all $i.$
	Since ${\bf Y}_n\to {\bf 0}$ and ${\bf 0}$ is and ICP for $G$, $ {\bf Y}_n=0$ for all $n\gg 1$.
	
	For the second statement, suppose $p_2$ is non-degenerate By Morse's lemma \cite{D73} for local fields $k$, ch.$k=0$\footnote{Duistermaat's proof goes through to the non-archimedean case without any change.}, there is a local diffeomorphism of $k^{m-1}$ fixing ${\bf 0}$ taking $G$ to $p_2$. But ${\bf 0}$ is an isolated CP for $p_2$, which makes it isolated for $G$. 
\end{proof}
\begin{lem} 
	The converse to the first statement of Lemma \ref{lem:ex3} is true if $F=X^{d-r}p_r+p_d$ ($r\geq 2$). 
\end{lem}
\begin{proof}
	We must show that $G=p_r+p_d$ has ${\bf 0}$ as an ICP if $(1,{\bf 0})$ is an ICP for $F$. Suppose ${\bf w}_n$ are CP's for $G=p_r+p_d$ with ${\bf w}_n\to {\bf 0}.$ Then, $G({w}_n)=F(1,{\bf w}_n)=0$  for all $n$, and $G_{i}({\bf w}_n)=\frac{\partial F}{\partial {Y}_i}(1,{\bf w}_n)=0$ for all $n.$
	Hence, $p_{r,i}({\bf w}_n)+p_{d,i}({\bf w}_n)=0 \text{ for all } n.$
	By Euler's theorem, $rp_r({\bf w}_n)+dp_d({\bf w}_n) =0 $ for all $n$. But, $p_r({\bf w}_n)+p_d({\bf w}_n)=0$ for all $n$ as well. So, $p_r({\bf w}_n)=p_d({\bf w}_n)=0$ for all $n$. Hence, $\frac{\partial F}{\partial X}(1,{\bf w}_n)= (d-r)p_r({\bf w}_n)=0 \text{ for all } n.$
	So $(1,{\bf w}_n)$ is a CP of $F$ for all $n$. As $(1, {\bf 0}) $ is assumed to be an ICP for $F$, ${\bf w}_n={\bf 0}$ for $n \gg 1$. So ${\bf 0}$ is an ICP for $F$.
\end{proof}
\subsection{Study of condition (\ref{eq:7.2}) for $F=X^{d-2}p_2+p_d$ where $G=p_2+p_d$ has ${\bf 0}$ as an ICP}
Let us consider $F=X^2+P_4(Y)$ where $P_4$ is a homogeneous quartic polynomial in $Y,Z$.
For this to have $(t,0,0)$ as and ICL we must have $(0,0)$ as an ICP for $G=Z^2+P_4(Y,Z)$.
\begin{lem}
	$G=Z^2+P_4(Y,Z)$ has ${\bf 0}$ as an ICP if and only if $Z^2 \not | P_4(Y,Z)$, i.e.,
	\begin{displaymath}
	P_4(Y,Z)= a_0 Y^4 + a_1 Y^3 Z + a_2 Y^2 Z^2 +a_3 Y Z^3 + a_4 Z^4
	\end{displaymath}
	where at least one of $a_0, a_1$ is nonzero. In this case ${\bf 0}$ is its only CP.
\end{lem}

\begin{proof}
	The equations which determine whether $(y,z)$ is a CP of $G$ are $z^2+P_4(y,z)=0,$ $\frac{\partial P_4}{\partial Y}(y,z)=0,$ and $2Z+\frac{\partial P_4}{\partial Z}(y,z)=0.$ 
	
	From the second and third equations just defined, using Euler's theorem, $2z^2+4P_4(y,z)=0,$
	which implies $z^2=0$ and $P_4(yz)=0$.
	
	So the only critical points are of the form $(y,0).$ Then $(0,0)$ is certainly a CP. If $(y,0)$ is a critical point for some $y\neq 0$, then $4a_0y^3=0, \ a_1y^3=0$ which implies $a_0, a_1$ both vanish.
	The entire $Y$-axis consists of critical points, and so for $(0,0)$ to be an ICP, at least one of $a_0, a_1 \neq 0.$, in which case $(0,0)$ is the only CP.
	
	We consider the cases ${\rm (I)} \ a_0\neq 0$ and ${\rm (II)} \ a_0=0, \ a_1 \neq 0.$
	We consider case (I). We shall now verify that $\inf_{\norm{\bf u}>1}\norm{\nabla F({\bf u})} > 0$ if ${\bf u}\in L_c, \norm{{\bf u}}\geq 1$. Assume $F=X^2 Z^2 +P_4(Y,Z)$, and in view of lemma \ref{lem:7.3}, choose a sequence $(x_n,y_n,z_n)$ such that $x_n\to \infty, \ y_n/x_n \to 0, \ z_n/x_n \to 0$ and: (i) $ x_n^2 y_n^2 +P_4(y_n,z_n) = c$, (ii) $\frac{\partial F}{\partial X}= 2x_n z_n^2 \to 0$, (iii) $\frac{\partial P_4}{\partial Y}(y_n,z_n) \to 0$,  (iv) $2x^2_n z_n + \frac{\partial P_4}{\partial Z}(y_n,z_n)\to 0.$
	
	From (ii) we get $z_n\to 0$.
	Assuming we are in case (I), {$y_n$ is bounded.} Otherwise by passing to a subsequence we may assume, $y_n\to \infty$ giving ${\partial P_4}/{\partial Y}(y_n,z_n)= 4a_0y^3_n +3a_1y_n^2z_n + \ldots\to 0.$ 
	If $a_0\neq 0$, then $\frac{\partial P_4}{\partial Y}(y_n,z_n)= 4a_0y_n^3(1+o(z_n/y_n))\to \infty,$ which is a contradiction. But if $\eta\neq 0$ is a limit point of $y_n$, then
	\begin{displaymath}
	\frac{\partial P_4}{\partial Y}(z_n,y_n) \to 4 a_0 \eta^3 \neq 0
	\end{displaymath}
	which is a contradiction. So, $y_n\to 0$ necessarily. Then, $P_4(y_n,z_n)\to 0$ and $\frac{\partial P_4}{\partial Z}(y_n,z_n)\to 0 $. Hence by (iv), $x^2_n z_n \to 0,$ by (i) $x_n^2 z_n^2 \to c\neq 0,$ a contradiction. This finishes case (I).
	
	Assuming we are in case (II), $a_0=0, \ a_1\neq 0,$ we claim $y_n\to \infty.$ Otherwise, by passing to a subsequence, we may assume $y_n\to \eta.$ Then $P_4(y_n,z_n) = a_1y^3_n z_n+\ldots$ so that $P_4(y_n, z_n)\to 0.$
	Hence, $x_n^2 z_n^2\to c.$
	But $\frac{\partial P_4}{\partial Z}(y_n,z_n)=a_1y_n^3+\ldots \to a_1\eta^3.$ Hence, by (iv), $x_n^2 z_n = o(1).$ So, as $z_n\to 0,$ we have $x_n^2 y_n^2\to 0$. Hence, $c=0$ is a contradiction.
	
	We are left with the case, $x_n\to \infty,$ $y_n \to \infty,$ $z_n \to 0,$ $\frac{y_n}{x_n} \frac{z_n}{x_n}\to 0,$ and $P_4(Y,Z)=a_1Y^3 Z + \ldots,$ for $a_1\neq 0.$
	But ${\partial P_4}/{\partial Y}(y_n,z_n)= 3 a_1 y^2_n z_n(1+o({z_n}/{y_n}))\to 0$
	if and only if $y^2_n z_n\to 0.$ In this case may we have a counterexample to statement (\ref{eq:7.2}). Remark \ref{rem:4.15} gives an example of this kind. Note that case (I) is generic among the families we consider.
\end{proof}
\bibliography{measures_bib}

\begin{thebibliography}{Mum95}

\bibitem[Art67]{artin67}
Emil Artin.
\newblock {\em Algebraic numbers and algebraic functions}.
\newblock Gordon and Breach Science Publishers, New York-London-Paris, 1967.

\bibitem[Bol90]{bollaerts1990estimate}
Dirk Bollaerts.
\newblock An estimate of approximation constants for p-adic and real varieties.
\newblock {\em manuscripta mathematica}, 69(1):411--442, 1990.

\bibitem[Dui73]{D73}
J.~J. Duistermaat.
\newblock {\em Fourier integral operators}.
\newblock Courant Institute of Mathematical Sciences, New York University, New
  York, 1973.
\newblock Translated from Dutch notes of a course given at Nijmegen University,
  February 1970 to December 1971.

\bibitem[FM78]{FM1977}
William Fulton and Robert MacPherson.
\newblock Defining algebraic intersections.
\newblock In {\em Algebraic geometry ({P}roc. {S}ympos., {U}niv. {T}roms\o,
  {T}roms\o, 1977)}, volume 687 of {\em Lecture Notes in Math.}, pages 1--30.
  Springer, Berlin, 1978.

\bibitem[GP74]{guillemin1974difierential}
Victor Guillemin and Alan Pollack.
\newblock Differential topology.
\newblock {\em Englewood Cliffs, New Jersey: PrenticeHaIl}, 1974.

\bibitem[Gre66]{greenberg1966rational}
Marvin~J Greenberg.
\newblock Rational points in henselian discrete valuation rings.
\newblock {\em Publications Math{\'e}matiques de l'IH{\'E}S}, 31:59--64, 1966.

\bibitem[GS64]{gelfand1964}
I.~M. Gel'fand and G.~E. Shilov.
\newblock {\em Generalized functions. {V}ol. {I}: {P}roperties and operations}.
\newblock Translated by Eugene Saletan. Academic Press, New York-London, 1964.

\bibitem[GW10]{gor2010}
Ulrich G{\"o}rtz and Torsten Wedhorn.
\newblock {\em Algebraic geometry {I}}.
\newblock Advanced Lectures in Mathematics. Vieweg + Teubner, Wiesbaden, 2010.
\newblock Schemes with examples and exercises.

\bibitem[HC57]{Har1}
Harish-Chandra.
\newblock Fourier transforms on a semisimple {L}ie algebra. {I}.
\newblock {\em Amer. J. Math.}, 79:193--257, 1957.

\bibitem[HC64]{Har3}
Harish-Chandra.
\newblock Invariant distributions on {L}ie algebras.
\newblock {\em Amer. J. Math.}, 86:271--309, 1964.

\bibitem[HC73]{Har2}
Harish-Chandra.
\newblock Harmonic analysis on reductive {$p$}-adic groups.
\newblock In {\em Harmonic analysis on homogeneous spaces ({P}roc. {S}ympos.
  {P}ure {M}ath., {V}ol. {XXVI}, {W}illiams {C}oll., {W}illiamstown, {M}ass.,
  1972)}, pages 167--192. Amer. Math. Soc., Providence, R.I., 1973.

\bibitem[Hec82]{Heck1982}
G.~J. Heckman.
\newblock Projections of orbits and asymptotic behavior of multiplicities for
  compact connected {L}ie groups.
\newblock {\em Invent. Math.}, 67(2):333--356, 1982.

\bibitem[H{\"o}r58]{hormander1958division}
Lars H{\"o}rmander.
\newblock On the division of distributions by polynomials.
\newblock {\em Arkiv f{\"o}r matematik}, 3(6):555--568, 1958.

\bibitem[IR78]{igusa1978lectures}
Jun-ichi Igusa and S~Raghavan.
\newblock {\em Lectures on forms of higher degree}.
\newblock Tata Institute of Fundamental Research Bombay, 1978.

\bibitem[Kos63]{Kos63}
Bertram Kostant.
\newblock Lie group representations on polynomial rings.
\newblock {\em Amer. J. Math.}, 85:327--404, 1963.

\bibitem[KV92]{Kolk}
Johan A.~C. Kolk and V.~S. Varadarajan.
\newblock Lorentz invariant distributions supported on the forward light cone.
\newblock {\em Compositio Math.}, 81(1):61--106, 1992.

\bibitem[{\L}oj59]{lojasiewicz1959}
Stanis{\l}aw {\L}ojasiewicz.
\newblock Sur le probl\`{e}me de la division.
\newblock {\em Studia Mathematica}, 18(1):87--136, 1959.

\bibitem[MO15]{MO2015}
David Mumford and Tadao Oda.
\newblock {\em Algebraic geometry. {II}}, volume~73 of {\em Texts and Readings
  in Mathematics}.
\newblock Hindustan Book Agency, New Delhi, 2015.

\bibitem[Mum95]{mumford1995algebraic}
David Mumford.
\newblock {\em Algebraic Geometry: Complex projective varieties. vol. 1},
  volume~1.
\newblock Springer Science \& Business Media, 1995.

\bibitem[RR72]{RangaRao72}
R.~Ranga~Rao.
\newblock Orbital integrals in reductive groups.
\newblock {\em Ann. of Math. (2)}, 96:505--510, 1972.

\bibitem[Sar42]{Sard42}
Arthur Sard.
\newblock The measure of the critical values of differentiable maps.
\newblock {\em Bull. Amer. Math. Soc.}, 48:883--890, 1942.

\bibitem[Ser09]{serre2009lie}
Jean-Pierre Serre.
\newblock {\em Lie algebras and Lie groups: 1964 lectures given at Harvard
  University}.
\newblock Springer, 2009.

\bibitem[Var77]{Var1}
V.~S. Varadarajan.
\newblock {\em Harmonic analysis on real reductive groups}.
\newblock Lecture Notes in Mathematics, Vol. 576. Springer-Verlag, Berlin-New
  York, 1977.

\bibitem[Var84]{var2}
V.~S. Varadarajan.
\newblock {\em Lie groups, {L}ie algebras, and their representations}, volume
  102 of {\em Graduate Texts in Mathematics}.
\newblock Springer-Verlag, New York, 1984.
\newblock Reprint of the 1974 edition.

\bibitem[Vog84]{Vog1984}
W.~Vogel.
\newblock {\em Lectures on results on {B}ezout's theorem}, volume~74 of {\em
  Tata Institute of Fundamental Research Lectures on Mathematics and Physics}.
\newblock Published for the Tata Institute of Fundamental Research, Bombay; by
  Springer-Verlag, Berlin, 1984.
\newblock Notes by D. P. Patil.

\bibitem[VW14]{virtanen2014elementary}
J.~Virtanen and D.~Weisbart.
\newblock Elementary particles on p-adic spacetime and temperedness of
  invariant measures.
\newblock {\em P-Adic Numbers, Ultrametric Analysis, and Applications},
  6(4):318--332, 2014.

\bibitem[Wei13]{weil2013basic}
Andr{\'e} Weil.
\newblock {\em Basic number theory.}, volume 144.
\newblock Springer Science \& Business Media, 2013.

\end{thebibliography}
\bibliographystyle{alpha}

\end{document}